\documentclass[12pt]{article}
\usepackage[english]{babel}

\usepackage{amsmath}
\usepackage{amssymb,amsthm}

\usepackage{bbm}

\numberwithin{equation}{section}

\newtheorem{theorem}{Theorem}

\newtheorem{lemma}{Lemma}[section]

\newtheorem{remark}[lemma]{Remark}
\newtheorem{proposition}{Proposition}
\newtheorem{notation}[lemma]{Notation}

\newtheorem{calcul}[lemma]{Calculation}

\renewcommand{\leq}{\leqslant}
\renewcommand{\geq}{\geqslant}

\newcommand{\Rot}{\mathcal{R}}

\makeatletter
\newcommand*{\IfItalic}{%
  \ifx\f@shape\my@test@it
    \expandafter\@firstoftwo
  \else
    \expandafter\@secondoftwo
  \fi
}
\newcommand*{\my@test@it}{it}
\makeatother

\newcommand{\myae}{\IfItalic{\emph{\mbox{\ae}}}{\mbox{\ae}}}

\newcommand\XLast{}
\newcommand\YLast{}

\newcommand\Vidr[4]{
\qbezier(#1,#2)(#1,#2)(#3,#4)
\renewcommand\XLast{#3}
\renewcommand\YLast{#4}}

\newcommand\VidrTo[2]{\Vidr{\XLast}{\YLast}{#1}{#2}}

\usepackage[all]{xy}

\usepackage{amscd}

\newcommand{\ex}[1]{\mathsf{E}\left[\,#1\,\right]}
\newcommand{\ind}[1]{\mathbbm{1}_{#1}}

\begin{document}

\begin{center}
{\Large  The derivative of the conjugacy\\ for the pair of
tent-like maps from an interval into itself}

{\large Makar Plakhotnyk\\
University of S\~ao Paulo, Brazil.\\

makar.plakhotnyk@gmail.com}\\
\end{center}

\begin{abstract}
We consider in this article the properties of the topological
conjugacy of the piecewise linear unimodal maps $g:\, [0,\,
1]\rightarrow [0,\, 1]$, all whose kinks belong to the complete
pre-image of $0$. We call such maps firm carcass maps. We prove
that every firm carcass maps $g_1$ and $g_2$ are topologically
conjugated. For the conjugacy $h$ such that $h\circ g_1 = g_2\circ
h$ we denote $\{ h_n, n\geq 1\}$ the piecewise linear
approximations of $h$, whose graphs connect the points $\{ (x,
h(x)),\ g_1^n(x)=0\}$. For any $x\in [0,\, 1]$ we reduce the
question about the value of $h'(x)$ to the properties of the
sequence $\{h_n'(x),\, n\geq 1\}$. We prove that each conjugacy of
firm carcass maps either has the length 2, or is piecewise
linear~\footnote{This work is partially supported by FAPESP (S\~ao
Paulo, Brazil).}~.
\end{abstract}

\section{Introduction}

The topological conjugation is a powerful tool for the
investigation of properties of one-di\-men\-sio\-nal dynamical
systems. The classical example in the pedagogic of one-dimensional
dynamical systems is the conjugacy of the maps
\begin{equation}\label{eq:1.1} x \mapsto 2x-|1-2x|
\end{equation} and $ f:\, x\mapsto 4x(1-x) $, which was stated at first
in~\cite{Ulam-1964-a}. Due to the form of the graph, the
map~\eqref{eq:1.1} is called \underline{\emph{tent map}}. The
mentioned example inspirits the desire to describe all the
functions, which are topologically conjugated to the tent map.
This description is also given in~\cite{Ulam-1964-a}. By the
definition of the topological conjugation, for the increasing
conjugacy $h:\, [0,\, 1]\rightarrow [0,\, 1]$, the map $g = h\circ
f\circ h^{-1}$ appears to be of the form
\begin{equation}\label{eq:1.2} g(x) = \left\{
\begin{array}{ll}g_l(x),& 0\leq x\leq
v,\\
g_r(x), & v\leq x\leq 1,
\end{array}\right.
\end{equation} where %
$v\in (0,\, 1)$, the function $g_l$ increase, the function $g_r$
decrease, and $$g(0)=g(1)=1-g(v)=0.$$

We will call \underline{\emph{unimodal map}} the mentioned $g$.

\begin{theorem}\label{th:1}\cite[p. 53]{Ulam-1964-a} %
The tent map~\eqref{eq:1.1} is topologically conjugated to the
unimodal map $g$ if and only if the complete pre-image of $0$
under the action of $g$ is dense in $[0,\, 1]$.
\end{theorem}

Remind, that the set $g^{-\infty}(a) = \bigcup\limits_{n\geq
1}g^{-n}(a)$, where $g^{-n}(a) = \left\{ x\in [0,\, 1]:\, g^n(x) =
a\right\}$ for all $n\geq 1$, is called the complete pre-image of
$a$ (under the action of the map $g$).

In spite of the elegance of the statement of Theorem~\ref{th:1},
it may appear to be quite difficult to resolve for a given
unimodal map $g$, whether or not $g^{-\infty}(0)$ is dense in
$[0,\, 1]$. Thus, it is natural to restrict the attention to some
subclasses of the unimodal maps.

The simplest partial case of the unimodal maps is when the graphs
of both $g_l$ and $g_r$ in~\eqref{eq:1.2} are segments of lines,
i.e. the map~\eqref{eq:1.2} is of the form
\begin{equation}\label{eq:1.3}
f_v(x) = \left\{\begin{array}{ll}
\frac{x}{v},& \text{if }0\leq x\leqslant v,\\
 \frac{1-x}{1-v},& \text{if }
v\leq x\leq 1,
\end{array}\right.
\end{equation} where $v\in (0,\, 1)$ is a
parameter. Due to~\cite{Skufca} and~\cite{Yong-Guo-Wang}, we will
call the map~\eqref{eq:1.3} a \underline{\emph{skew tent map}}.
Remark that $f_{1/2}$ is the tent map~\eqref{eq:1.1}. The
existence and the uniqueness of the conjugacy of $f_{v_1}$ and
$f_{v_2}$ of the form~\eqref{eq:1.3} for all distinct $v_1,\,
v_2\in (0,\, 1)$ was stated in~\cite{Skufca}. In other words,
there exists the unique continuous invertible solution $h:\, [0,\,
1]\rightarrow [0,\, 1]$ of the functional equation
\begin{equation}\label{eq:1.4} h\circ f_{v_1} = f_{v_2}\circ h.
\end{equation}

We have independently proved in~\cite{Plakh-Fedor-2014} the
conjugateness of the tent map and $f_v$ for all $v\in (0,\, 1)$.
Precisely, we have proved the following
\begin{proposition}\cite[Lema~10]{Plakh-Fedor-2014}\label{prop:1}
For every $v\in (0,\, 1)$ the set $f_v^{-\infty}(0)$ is dense in
$[0,\, 1]$, where $f_v$ is skew tent map.
\end{proposition}

The next property of the conjugacy from~\eqref{eq:1.4} is obtained
in~\cite{Skufca}.

\begin{proposition}\cite[Prop.~2]{Skufca}\label{prop:2}
If the derivative of the continuous invertible solution $h$
of~\eqref{eq:1.4} is finite\footnote{in fact, Proposition~2
of~\cite{Skufca} says that $h'(x)=0$ everywhere, where it exists.
Nevertheless, it follows from the proof of Proposition~2, that the
authors of~\cite{Skufca} assume that the derivative can be only
finite. In fact, they take an arbitrary point $x\in (0, 1)$ and
construct a sequence $k_{n}$ such that $x\in I_{n}
=\left[\frac{k_{n} }{2^{n} },\, \frac{k_{n} +1}{2^{n} } \right]$
and $p_{n} =h\left(\frac{k_{n} +1}{2^{n} }
\right)-h\left(\frac{k_{n} }{2^{n} } \right)$. After this they
claim that if $h'(x)$ exists and is non-zero, then $\frac{p_{n+1}
}{p_{n} } \to \frac{1}{2}$. But this is true only in the case
$h'(x)<\infty$. Following~\cite[Sect~92, 101]{Fihtengoltz}, we
will assume, that the derivative is a limit, which can be also
infinite and in this case so is the value of the derivative.} at
some $x\in [0,\, 1]$, then $h'(x)=0$.
\end{proposition}

This result is especially interesting, because Lebesgue's Theorem
(see~\cite{Lebeg} or~\cite[p. 15]{Riss}) claims that each
nondecreasing function has a finite derivative at every point with
the possible exception of the points of a set of measure zero. The
conjugacy $h$ of the maps $f_{v_1}$ and $f_{v_2}$ is also studied
in~\cite{Yong-Guo-Wang}.

\begin{proposition}\label{prop:3}\cite[Prop.~1]{Yong-Guo-Wang}
For any distinct $v_1,\, v_2\in (0, 1)$ the length of the graph of
the conjugacy $h$ of maps $f_{v_1}$ and $f_{v_2}$ of the
form~\eqref{eq:1.3} equals $2$.
\end{proposition}

Remark that the length of the graph, mentioned in
Proposition~\ref{prop:3}, is the maximum possible length of a
monotone $[0,\, 1]\rightarrow [0,\, 1]$ function.

We will call the map defined on the set $A\subseteq \mathbb{R}$
with values in $\mathbb{R}$, \underline{\emph{linear}}, it its
graph is a line (or a line segment). We will call the map
\underline{\emph{piecewise linear}}, if its domain can be divided
into finitely many intervals, such that on each of them the map is
linear. A point, where the piecewise linear map is not
differentiable, will be called a \underline{\emph{kink}}.

The piecewise linear unimodal map will be called
\underline{\emph{carcass map}}. A carcass map, all whose kinks
belong to the complete pre-image of $0$, will be called
\underline{\emph{a firm carcass map}}.

Let $g_1, g_2:\, [0, 1]\rightarrow [0, 1]$ be firm carcass maps.
For every $n\geq 1$ denote by $h_n:\, [0,\, 1]\rightarrow [0,\,
1]$ the increasing piecewise linear map such that $h_n(\,
g_1^{-n}(0)\, ) = g_2^{-n}(0)$ and all the kinks of $h_n$ belong
to $g_1^{-n}(0)$. We will prove that these $g_1$ and $g_2$ are
topologically conjugated, moreover, each of them is conjugated to
the tent map. We will prove that $h = \lim\limits_{n\rightarrow
\infty}h_n$, where the limit is considered point-wise, exists and
this limit function is the homeomorphism, which satisfies the
functional equation
\begin{equation}\label{eq:1.5}
h\circ g_1 = g_2\circ h.
\end{equation}

In fact, the idea of this construction appeared in the original
proof of Theorem~\ref{th:1}, given in~\cite{Ulam-1964-a}. We will
prove the following generalization our Proposition~\ref{prop:1}.

\begin{theorem}\label{th:2}
The complete pre-image of 0 under the action of every firm carcass
map is dense in $[0,\, 1]$.
\end{theorem}

Notice that Theorems~\ref{th:1} and~\ref{th:2} imply that every
firm carcass map is topologically conjugated to the tent-map.

For any $x\in [0,\, 1]$ denote
\begin{equation}\label{eq:1.6}L(x)=\left\{
\begin{array}{ll}
\lim\limits_{n\rightarrow \infty}h_n'(x-) & \text{if } x>0\\
\lim\limits_{n\rightarrow \infty}h_n'(x+) & \text{otherwise}
\end{array}\right.
\end{equation} and
\begin{equation}\label{eq:1.7}
R(x)=\left\{ \begin{array}{ll}
\lim\limits_{n\rightarrow \infty}h_n'(x+) & \text{if } x<1\\
\lim\limits_{n\rightarrow \infty}h_n'(x-) & \text{otherwise}
\end{array}\right.
\end{equation}

We will prove the following theorem.

\begin{theorem}\label{th:3}
Let $g_1$ and $g_2$ be firm carcass maps and let $h$ be the
conjugacy, which satisfies~\eqref{eq:1.5}.

1. If for at least one $x\in [0,\, 1]$ the derivative $h'(x)$
exists, is positive and finite, then $h$ is piecewise linear.

2. If $h$ is not piecewise linear, then $h'(x)$ exists if and only
if there exists $L(x)$, $R(x)$ and, moreover, $L(x) = R(x)$. In
this case $h'(x) = L(x)$.
\end{theorem}

For every $a,\, b\in \{0,\, 1\}$ denote
\begin{equation}\label{eq:1.8}
\myae_v(a,\, b) = \left\{ \begin{array}{ll} v & \text{if
}a=b,\\
1-v& \text{if }a\neq b\end{array}\right. \end{equation}

Theorem~\ref{th:3} can be specified in the case of the conjugation
of the tent map with a skew tent map.

\begin{theorem}\label{th:4}
For any $v\in (0,\, 1)$ let $h$ be the conjugacy of the tent map
$f$ and the map $f_v$ of the form~\eqref{eq:1.3}, i.e.
$$f\circ h = h\circ f_v.
$$ The %
derivative $h'(x)$ exists if and only if there exits the limit
$\myae_\infty(x) =\prod\limits_{k=2}^\infty (2\, \myae_v(x_{k},
x_{k-1}))$, where $ 0.x_1x_2\ldots $ is the binary expansion of
$x$. Moreover, in this case $h'(x) = \myae_\infty(x)$.
\end{theorem}

We have obtained the partial cases of Theorem~\ref{th:4} in our
previous works. We have proved Theorem~\ref{th:4} for the binary
finite numbers in~\cite{Visnyk} and we also have proved
Theorem~\ref{th:4} for the rational numbers in~\cite{Studii}.

Notice, that Proposition~\ref{prop:2} is a simple corollary of
Theorem~\ref{th:4}.

Proposition~\ref{prop:3} can be generalized for the firm carcass
maps as follows.

\begin{theorem}\label{th:5}
Let $h$ be the conjugacy of any firm carcass maps. If $h$ is not
piecewise linear, then the length of the graph of $h$ equals $2$.
\end{theorem}

Our work consists of eight sections except introduction. In
section~\ref{sec:2} we present some basic facts about topological
conjugacy and the tent map. Section~\ref{sec:3} is devoted to the
construction of the linear approximations of the conjugacy of
unimodal functions, which was obtained at first in the original
proof of Theorem~\ref{th:1} and we generalize this Theorem. We use
this approximation in our further reasonings. In
section~\ref{sec:4} we construct the generalization of the binary
expansion of a number, precisely for any $x\in [0,\, 1]$ and the
map $g$ of the form~\eqref{eq:1.2} we construct an infinite
sequence $\{ x_i,\, i\geq 1, x_i\in \{0; 1\}\}$, which determines
the number $x$. We also specify the properties of the obtained
sequence for firm carcass maps and prove Theorem~\ref{th:2}. In
Section~\ref{sec:5} we study the derivatives of the conjugacy of
firm carcass maps, precisely we prove Theorem~\ref{th:3} there. In
Section~\ref{sec:6} we prove Theorem~\ref{th:5}. In We devote
Section~\ref{sec:7} to the topological conjugacy of the tent map
and a skew tent map. We derive Theorem~\ref{th:4} from
Theorem~\ref{th:3} there give an alternative proof of
Theorem~\ref{th:5}. We state some hypothesis for the further
research in Section~\ref{sec:8}.

\section{Basic facts and properties}\label{sec:2}

\subsection{Basic facts about tent map}

We will state some, quite clear, properties of tent maps in his
section. These properties are almost evident, whenever they are
already formulated from one hand and, maybe, all of them are
mentioned in some text books on the Theory of Dynamical Systems as
illustrations, or examples. Nevertheless, we want to state these
properties explicitly, because we will generalize them later in
Section~\ref{sec:4} for an arbitrary unimodal map.

\begin{remark}\label{rem:2.1}
Notice that we can rewrite the formula~\eqref{eq:1.1} as $$ f(x) =
\left\{
\begin{array}{ll}
2x & \text{if }\, 0\leq x\leq 1/2,\\
2-2x & \text{if }\, 1/2\leq x\leq 1.
\end{array}\right.
$$
\end{remark}

Remark~\ref{rem:2.1} provides the rule to construct the binary
expansion of $f(x)$ by the binary expansion of $x\in [0, 1]$.

\begin{notation}\label{not:2.2}
Denote $\Rot(t) = 1-t$ for $t\in \{0; 1\}$.
\end{notation}

\begin{remark}\label{rem:2.3}
Let $f$ be the tent map and \begin{equation}\label{eq:2.1} x =
0.x_1x_2\ldots\, x_n\ldots\end{equation} %
be the binary expansion of an arbitrary $x\in [0,\, 1]$. Then the
binary expansion of $f(x)$ is
$$ f(x)= \left\{
\begin{array}{ll}
0.x_2x_3\ldots x_n\ldots, & \text{if }x_1 =
0,\\
0.\Rot(x_2)\Rot(x_3)\ldots \Rot(x_n)\ldots, & \text{if }x_1 = 1.
\end{array}\right.
$$
\end{remark}

Notice, that we understand~\eqref{eq:2.1} as
\begin{equation}\label{eq:2.2} x = \sum\limits_{i=1}^\infty
x_i2^{-i},
\end{equation} %
precisely we will understand~\eqref{eq:2.1} and~\eqref{eq:2.2} as
two different forms to write the same fact. We will
use~\eqref{eq:2.1} because we think that this form of
representation is more visual than~\eqref{eq:2.2}.

Remark~\ref{rem:2.3} provides the description of the set
$f^{-n}(0)$ for all $n\geq 1$.

\begin{remark}\label{rem:2.4}
For every $n> 1$ the set $\{ x<1:\, f^n(x) =0\}$ consists of all
the $x\in [0, 1]$ with the binary expansion $$ x =
0.x_1x_2\ldots\, x_{n-1}.
$$
\end{remark}

Next, Remark~\ref{rem:2.4} can be rewritten as follows.

\begin{remark}\label{rem:2.5}
For every $n\geq 1$ we have that $$f^{-n}(0) = \left\{
\frac{k}{2^{n-1}},\, 0\leq k\leq 2^{n-1}\right\}.$$
\end{remark}

\begin{remark}\label{rem:2.6} The graph of the $n$th iteration $f^n$
has the following properties:

1. $f^n(0)=0$, i.e. the graph passes through origin.

2. The graph consists of $2^n$ line segments, whose tangents are
either $2^n$, or $-2^n$.

3. Each maximal part of monotonicity of $f^n$ connects the line
$y=0$ and $y=1$.

4. Let $x_1,\, x_2$ be such that $\{ f^n(x_1);\, f^n(x_2)\} =
\{0;\, 1\}$ and $f^n$ is monotone on $[x_1,\, x_2]$. Then
$f^{n+1}(x_1) = f^{n+1}(x_2) = 0$. Precisely, for $x_3 =
\frac{x_1+x_2}{2}$ we have that $f^{n+1}(x_3) = 1$ and, moreover,
$f^{n+1}$ increase on $[x_1,\, x_3]$ and decrease on $[x_3,\,
x_2]$, being linear at each of these intervals.
\end{remark}

\subsection{Properties of topological conjugation}

The change of coordinates is one of the classical illustration
(explanation) of what topological conjugation is. Thus, it
preserves a lot of properties of maps and points. From another
hand, if topologically conjugated maps are ``similar'' in some
sense, then local properties of the conjugacy can be globalized.
Moreover, this globalization still holds in the case, when we are
talking about semi conjugation.

\begin{lemma}
Let $g_1$ and $g_2$ be unimodal maps and let a solution $h$
of~\eqref{eq:1.5} be continuous. Then:

1. For any fixed point $x$ of $g_1$, the point $h(x)$ is a fixed
point of $g_2$.

2. If $h$ is invertible then for any periodical point $x$ of $g_1$
of period $n$, the point $h(x)$ is periodical point of $g_2$ of
period $n$.
\end{lemma}

The next fact us quite technical, but is important for the further
corollaries.

\begin{lemma}\label{lema:2.8}
Suppose that $h$ is a continuous solution of~\eqref{eq:1.5}. Let
$a, b\in [0,\, 1]$ be such that $g_1$ and $g_2$ are linear on $[a,
b]$ and $[h(a), h(b)]$ respectively. Then for every $c\in (a, b)$
$$ \frac{h(c)-h(a)}{c-a}\cdot \frac{b-a}{h(b)-h(a)} =
\frac{(h\circ g_1)(c)-(h\circ g_1)(a)}{g_1(c) -g_1(a)}\cdot
\frac{g_1(b) -g_1(a)}{(h\circ g_1)(b)-(h\circ g_1)(a)}
$$ holds.
\end{lemma}

\begin{proof}
It follows from~\eqref{eq:1.5} that $(h\circ g_1)(c)-(h\circ
g_1)(a) = (g_2\circ h)(c)-(g_2\circ h)(a)$, and, by linearity of
$g_2$ on $[h(a), h(b)]$, $(g_2\circ h)(c)-(g_2\circ h)(a) =
g_2'(c)\cdot (h(c)-h(a))$, whence $$ (h\circ g_1)(c)-(h\circ
g_1)(a) = g_2'(c)\cdot (h(c)-h(a)).
$$ Analogously, $$(h\circ g_1)(b)-(h\circ
g_1)(a) = g_2'(c)\cdot (h(b)-h(a)).
$$

By linearity of $g_1$ on $[a, b]$ obtain $$ g_1(c) -g_1(a) =
g_1'(c)\cdot (c-a)
$$ and $$
g_1(b) -g_1(a) = g_1'(c)\cdot (b-a).
$$

Now, $$ \frac{(h\circ g_1)(c)-(h\circ g_1)(a)}{g_1(c)
-g_1(a)}\cdot \frac{g_1(b) -g_1(a)}{(h\circ g_1)(b)-(h\circ
g_1)(a)} =$$ $$= \frac{g_2'(c)\cdot (h(c)-h(a))}{g_1'(c)\cdot
(c-a)}\cdot \frac{g_1'(c)\cdot (b-a)}{g_2'(c)\cdot (h(b)-h(a))}
$$ and we are done.
\end{proof}

The next lemma follows from Lemma~\ref{lema:2.8}.

\begin{lemma}\label{lema:2.9}
Let $g_1$ and $g_2$ be carcass maps and let solution $h$
of~\eqref{eq:1.5} be continuous. Then:

1. If $h$ is constant on some interval, then $h$ is constant in
the entire $[0, 1]$.

2. If $h$ is linear on some interval, then $h$ is piecewise linear
in the entire $[0, 1]$.

3. If $h$ is differentiable on some interval, then $h$ in
piecewise differentiable on the entire $[0, 1]$ (i.e. $h$ is
differentiable everywhere on $[0, 1]$ except, possibly, finitely
many points).
\end{lemma}

\section{The Stanislaw Ulam's construction}\label{sec:3}

We will generalize Theorem~\ref{th:1} in this section. We will
also introduce some notations during the proof, which are
necessary for our further reasonings. The construction of the
sequence $\{ h_n,\, n\geq 1\}$ below is the simple generalization
of the proof of Theorem~\ref{th:1}, given in~\cite{Ulam-1964-a}.

The following remark is the generalization of
Remark~\ref{rem:2.6}.

\begin{remark}\label{rem:3.1}
The graph of the $n$th iteration $g^n$ of an arbitrary unimodal
function $g$ has the following properties:

1. The graph consists of $2^n$ monotone curves.

2. Each maximal part of monotonicity of $g^n$ connects the line
$y=0$ and $y=1$.

3. If $x_1,\, x_2$ are such that $\{ g^n(x_1);\, g^n(x_2)\} =
\{0;\, 1\}$ and $g^n$ is monotone on $[x_1,\, x_2]$, then
$g^{n+1}(x_1) = g^{n+1}(x_2) = 0$ and there is $x_3\in (x_1,\,
x_2)$ such that $g^{n+1}(x_3) = 1$. Moreover in this case
$g^{n+1}$ increase on $[x_1,\, x_3]$ and decrease on $[x_3,\,
x_2]$.
\end{remark}

Since, by Remark~\ref{rem:3.1}, the set $g^{-n}(0)$ consists of
$2^{n-1}+1$ points, then the notation follows.

\begin{notation}
For every unimodal map $g:\, [0,\, 1]\rightarrow [0,\, 1]$ and for
every $n\geq 1$ denote $\{ \mu_{n,k}(g),\, 0\leq k\leq 2^{n-1}\}$
such that $g^n(\mu_{n,k}(g))=0$ and $\mu_{n,k}(g)<\mu_{n,k+1}(g)$
for all $k$.
\end{notation}

\begin{remark}\label{rem:3.3}
Notice that $\mu_{n,k}(g) = \mu_{n+1,2k}(g)$ for all $k,\, 0\leq k
\leq 2^{n-1}$.
\end{remark}

\begin{calcul}\label{calc:3.4}
For each unimodal map $g$, every $n\geq 2$ and $k,\, 0\leq k\leq
2^{n-2}$ the equalities
\begin{equation}\label{eq:3.1} g(\, \mu_{n,k}(g)\,
)=\mu_{n-1,k}(g)\end{equation} and
\begin{equation}\label{eq:3.2}g(\, \mu_{n,k}(g)\, )= g(\,
\mu_{n,2^{n-1}-k}(g)\, )
\end{equation} hold.
\end{calcul}

For every two unimodal maps $g_1,\, g_2:\, [0,\, 1]\rightarrow
[0,\, 1]$ and every $n\in \mathbb{N}$ define the map
$\widehat{h}_n:\, g_1^{-n}(0)\rightarrow g_2^{-n}(0)$ by
\begin{equation}\label{eq:3.3}
\widehat{h}_n(\mu_{n,k}(g_1)) =\mu_{n,k}(g_2)\end{equation} for
all $k,\, 0\leq k\leq 2^{n-1}$.

\begin{lemma}\label{lema:3.5}
For every maps $g_1,\, g_2:\, [0,\, 1]\rightarrow [0,\, 1]$ the
map $\widehat{h}_n:\, g_1^{-n}(0)\rightarrow g_2^{-n}(0)$, defined
by~\eqref{eq:3.3}, satisfies the equation
\begin{equation}\label{eq:3.4} \widehat{h}_n\circ g_1 = g_2\circ
\widehat{h}_n.
\end{equation}
\end{lemma}

\begin{proof}
By Calculation~\ref{calc:3.4}, for every $k,\, 0\leq k\leq
2^{n-2}$ it follows from~\eqref{eq:3.1} that
\begin{equation}\label{eq:3.5}\left\{ \begin{array}{l}
\widehat{h}_n(\mu_{n,k}(g_1)) = \widehat{h}_n(\mu_{n-1,2k}(g_1)) =
\mu_{n-1,2k}(g_2)\\
g_2(\widehat{h}_n(\mu_{n,k}(g_1)\, )) = g_2(\, \mu_{n,k}(g_2)\, )
= \mu_{n-1,2k}(g_2),
\end{array}\right.\end{equation} and it follows from~\eqref{eq:3.2}
that \begin{equation}\label{eq:3.6}\left\{ \begin{array}{l}
\widehat{h}_n(\mu_{n,2^{n-1}-k}(g_1)) =
\widehat{h}_n(\mu_{n-1,2k}(g_1)) =
\mu_{n-1,2k}(g_2)\\
g_2(\widehat{h}_n(\mu_{n,2^{n-1}-k}(g_1)\, )) = g_2(\,
\mu_{n,k}(g_2)\, ) = \mu_{n-1,2k}(g_2).
\end{array}\right.\end{equation}

Now~\eqref{eq:3.5} and~\eqref{eq:3.6} imply~\eqref{eq:3.4}.
\end{proof}

\begin{lemma}\label{lema:3.6}
Suppose that unimodal maps $g_1,\, g_2:\, [0,\, 1]\rightarrow
[0,\, 1]$ are topologically conjugated and $h:\, [0,\,
1]\rightarrow [0,\, 1]$ is the conjugacy such that~\eqref{eq:1.5}
holds. Then
\begin{equation}\label{eq:3.7} h(\mu_{n,k}(g_1)) =\mu_{n,k}(g_2)
\end{equation} for all $n\geq 1$ and $k,\, 0\leq k\leq 2^{n-1}$.
\end{lemma}

\begin{proof}
By~\eqref{eq:1.5}, the number $h(0)$ is a fixed point of $g_2$.
Since $h(0)\in \{0;\, 1\}$, then $h(0)=0$. Thus, $h$ increase. It
follows by induction on $n$ from~\eqref{eq:1.5} that $ h\circ
g_1^n = g_2^n\circ h $ for all $n\geq 1$. The obtained equality
and $h(0)=0$ imply that $h(\, g_1^{-n}(0)\, ) = g_2^{-n}(0)$,
whence~\eqref{eq:3.7} follows.
\end{proof}

The next theorem is the generalization of Theorem~\ref{th:1}.

\begin{theorem}[Ulam's Theorem]
Let $g_1,\, g_2:\, [0,\, 1]\rightarrow [0,\, 1]$ be unimodal maps
and suppose that $g_1^{-\infty}(0)$ is dense in $[0,\, 1]$. Then
$g_1$ and $g_2$ are topologically conjugated if and only if
$g_2^{-\infty}(0)$ is dense in $[0,\, 1]$. Moreover, in this case
the conjugacy is unique.
\end{theorem}

\begin{proof}
For every $n\geq 1$ define $\widehat{h}_n:\,
g_1^{-n}(0)\rightarrow g_2^{-n}(0)$ by~\eqref{eq:3.3}.

Suppose that $h$ is the topological conjugacy of $g_1$ and $g_2$.
By Lemma~\ref{lema:3.6} the map $h$ coincides with $\widehat{h}_n$
on the domain of $\widehat{h}_n$, thence the density of
$g_2^{-\infty}(0)$ follows from the continuity of $h$.

Now assume that $g_2^{-\infty}(0)$ is dense in $[0,\, 1]$. Define
$h:\, [0,\, 1]\rightarrow [0,\, 1]$ by $h =
\lim\limits_{n\rightarrow \infty}h_n,$ where the limit is
considered point-wise. The function $h$ is well-defined due to the
density of $g_1^{-\infty}(0)$ in $[0,\, 1]$ and it is continuous
due to the density of $g_2^{-\infty}(0)$ in $[0,\, 1]$. Also
denote $\widehat{g}_n = h_n\circ g_1\circ h_n^{-1}.$ By
Lemma~\ref{lema:3.5} obtain that $\widehat{g}_n(x)=g_2(x)$ for all
$x\in g_2^{-n}(0)$ and~\eqref{eq:1.5} follows.
\end{proof}

\begin{remark}\label{rem:3.7}
It follows from~\eqref{eq:3.7} that the conjugacy of maps $g_1,\,
g_2$ of the form~\eqref{eq:1.2} increase.
\end{remark}

Remark, that the sequence $\{h_n,\, n\geq 1\}$ for the
approximation of the conjugacy of maps $f_v$ of the
form~\eqref{eq:1.3} also appeared in~\cite{Yong-Guo-Wang}. These
maps are denoted by $T_c$ in~\cite{Yong-Guo-Wang}, where $c\in
(0,\, 1)$ means the same as $v\in (0,\, 1)$ in our notations. The
solution $\varphi$ of the functional equation $\varphi\circ
T_{c_1} = T_{c_2}\circ \varphi$ is found in~\cite{Yong-Guo-Wang}
as the limit of the sequence $\{\varphi_n,\, n\geq 0\}$, where
$\varphi_0(x) = x$ for all $x\in [0,\, 1]$ and
\begin{equation}\label{eq:3.8} \varphi_{n+1}(x) = \left\{
\begin{array}{ll}
c_2\varphi_n\left(\frac{x}{c_1}\right) & \text{if } 0\leq x\leq
c_1,\\
(c_2-1)\varphi_n\left(\frac{x-1}{c_1-1}\right)+1 & \text{if
}c_1<x\leq 1.
\end{array}\right.
\end{equation}

Notice, that the sequence of functions $\{ h_n,\, n\geq 1\}$
satisfies~\eqref{eq:3.8}, whence $\varphi_n = h_{n+1}$ for all
$n\geq 0$.

The sequence~\eqref{eq:3.8} is considered in~\cite{Skufca} too for
the unimodal (not necessary piecewise linear) maps $g_1$ and
$g_2$, where, additionally, the peak of $g_1$ is $1/2$.
By~\cite[Lema~3]{Skufca} for any such $g_1$ and $g_2$ there is a
unique fixed element $h$ of~\eqref{eq:3.8}, which is bounded in
the Banach space with the norm $|h| = \sup\limits_{x\in [0,
1]}|h(x)|$. This fact is independent on wether of not $g_1$ and
$g_2$ are topologically conjugated. Clearly, if $g_1$ and $g_2$
are not topologically conjugated, then the fixed element $h$
of~\eqref{eq:3.8} is not a homeomorphism.

\section{The map-expansion of a number}\label{sec:4}

\subsection{General properties}

The construction below is the generalization of the binary
expansion of a number. By any given $x\in [0,\, 1]$ and by a
unimodal map $g$ we will construct an infinite sequence $\{x_i,
i\geq 1\}$ with $x_i\in \{0;\, 1\}$ for all $i\geq 1$, and denote
\begin{equation}\label{eq:4.1}
k_n = \sum\limits_{i=0}^{n}x_i2^{n-i}
\end{equation} for all $n\geq 1$.

\begin{lemma}\label{lema:4.1}
For every $x\in (0,\, 1)\setminus g^{-\infty}(0)$ there exists the
infinite sequence $\{x_i,\, i\geq 1\}$ with the following
properties:

1. $x_i\in \{0,\, 1\}$ for all $i\geq 1$,

2. $x\in (\mu_{n+1,k_n},\, \mu_{n+1,k_n+1})$, where $k_n$ is
defined by~\eqref{eq:4.1}.
\end{lemma}

\begin{proof}
The sequence $\{k_n,\, n\geq 1\}$ such that $x\in
(\mu_{n+1,k_n},\, \mu_{n+1,k_n+1})$ exists, because $g^n(x)\neq 0$
for all $n$. For any $n\geq 1$ assume that~\eqref{eq:4.1} holds.
Then it follows from Remark~\ref{rem:3.3} that $x\in
(\mu_{n+2,2k_n},\, \mu_{n+2,2k_n+2})$, whence either
$k_{n+1}=2k_n$, or $k_{n+1}=2k_n+1$. Thus, there exists $x_n\in
\{0;\, 1\}$ such that $k_{n+1} = 2\cdot k_n + x_n,$ whence %
$ k_{n+1} = x_n + 2\cdot\sum\limits_{i=0}^{n} x_i2^{n-i} =
\sum\limits_{i=0}^{n+1} x_i2^{n+1-i}$ and we are done.
\end{proof}

\begin{notation}\label{not:4.2}
For every $x\in [0,\, 1]$ construct the infinite sequence
$\{x_k,\, k\geq 0\}$ as follows:

1. If $x =1$, then set $x_k=1$ for all $k\geq 1$.

2. If $x<1$, but $x\in g^{-n}(x)=0$ for some $n\geq 1$, then set
$x_1,\ldots, x_{n-2}$ to be such that the number $k$ from the
equality $x = \mu_{n+1,k}$, has the binary expansion $k =
\sum\limits_{i=1}^{n}x_i2^{n-i}$; and set $x_k=0$ for all $k> n$.

3. If $g^n(x)\neq 0$ for all $n\geq 1$, then denote $\{ x_k,\,
k\geq 1\}$ the sequence from Lemma~\ref{lema:4.1}.

It will be convenient for uss to add $x_0=0$ to the left of the
sequences from Notation~\ref{not:4.2}. Thus, call the constructed
sequence \textbf{the $g$-expansion} of $x$ and write
\begin{equation}\label{eq:4.2}
x \stackrel{g}{\longrightarrow} x_0x_1\ldots
x_n\ldots\ .
\end{equation}

\begin{remark}
The $g$-expansion of any $x\in [0, 1]$ is unique if and only if
$g^{-\infty}(0)$ is dense in $[0, 1]$. Remind that, by Ulam's
Theorem~\ref{th:1}, in this case the map $g$ is topologically
conjugated to the tent map.
\end{remark}

Thus, if $g^{-\infty}(0)$ is dense in $[0, 1]$, then write
\begin{equation}\label{eq:4.3} x \stackrel{g}{\longleftrightarrow}
x_0x_1\ldots x_n\ldots
\end{equation} for %
the $g$-expansion of $x$. Say that $x$ is $g$-\textbf{finite}, if
there is $k$ such that $x_i=0$ for all $i>k$. Otherwise say that
$x$ is $g$-\textbf{infinite}. Analogously to rational numbers,
assume that the $g$-expansion of a $g$-finite number is finite
(i.e. does not contain the infinite series of zeros).
\end{notation}

\begin{remark}
Remark, that if $g$ is the tent map~\eqref{eq:1.1}, then
$g$-expansion is the classical binary expansion of a number.
\end{remark}

\begin{remark}
Notice, that for numbers $x,\, y\in [0,\, 1]$ the inequality
$x\leq y$ holds if and only if the $g$-expansion of $x$ is $\leq$
than $g$-expansion of $y$ in the natural lexicographical order.
\end{remark}

In spite of non-uniqueness of $g$-expansion in general, the
analogue of Remark~\ref{rem:2.3} holds for $g$-expansion.

\begin{remark}
Let~\eqref{eq:4.2} be a $g$-expansion of $x\in [0,\, 1]$. Then
\begin{equation}\label{eq:4.4} g(x) \stackrel{g}{\longrightarrow} \left\{
\begin{array}{ll}
0.x_2x_3\ldots x_n\ldots, & \text{if }x_1 =
0,\\
0.\Rot(x_2)\Rot(x_3)\ldots \Rot(x_n)\ldots, & \text{if }x_1 = 1,
\end{array}\right.
\end{equation} where %
$\Rot$ is defined in Notation~\ref{not:2.2}. Moreover,
if~\eqref{eq:4.3} holds, then ``$\stackrel{g}{\longrightarrow}$''
is ``$\stackrel{g}{\longleftrightarrow}$'' in~\eqref{eq:4.4}.
\end{remark}

\subsection{Specification for firm carcass map}

Let $g$ be a firm carcass map, which will be fixed till the end of
this section. Denote by $n_0$ the minimal natural number such that
$g^{n_0}(x) = 0$ for each kink $x$ of $g$.

\begin{notation}\label{not:4.7}
For any $n\geq 1$ and $k,\, 0\leq k< 2^{n-1}$ denote $I_{n,k} =
(\mu_{n,k},\, \mu_{n,k+1})$ and $\# I_{n,k} =
\mu_{n,k+1}-\mu_{n,k}$.
\end{notation}

\begin{remark}
For every $k,\, 0\leq k<2^{n_0-1}$ the graph of the function $g$
on $I_{n_0,k}$ is a segment of a line.
\end{remark}

\begin{notation}\label{not:4.9}
For any $n\geq 1$ and $k,\, 0\leq k< 2^n$ denote
$$\delta_{n,k} =
\frac{\mu_{n+1,2k+1}-\mu_{n,k}}{\mu_{n,k+1}-\mu_{n,k}}.$$
\end{notation}

\begin{remark}\label{rem:4.10}
Suppose that $I_{n+2,p}\subseteq I_{n+1,k}$, where~\eqref{eq:4.1}
is the binary expansion of $k$. By Remark~\ref{rem:3.3}, there
there exists $x_{n+1}\in \{0; 1\}$ such that $p =2k+x_{n+1}$.
Thus, by Notation~\ref{not:4.9},
$$ \# I_{n+2,p} = \# I_{n+1,k}\cdot
\Rot^{x_{n+1}}(\delta_{n+1,k})).
$$
\end{remark}

\begin{remark}\label{rem:4.11}
For any $n\geq n_0-1$ and $k,\, 0\leq k< 2^{n-1}$ denote
$I_{n-1,k^*} =g(I_{n,k})$. Then
$$
\delta_{n-1,k^*} = \left\{ \begin{array}{ll} \delta_{n,k}&
\text{if }I_{n,k}\subset I_{2,0},\\
1-\delta_{n,k}& \text{otherwise.}
\end{array}\right.
$$
\end{remark}

Using Notation~\ref{not:2.2}, we can rewrite Remark~\ref{rem:4.11}
as follows.

\begin{remark}\label{rem:4.13}
For any $n\geq n_0$ and $k,\, 0\leq k< 2^n$ with the binary
expansion~\eqref{eq:4.1} denote $I_{n,k^*} = g(I_{n+1,k})$. Then
$$ \delta_{n+1,k} = \Rot^{x_n}(\delta_{n,k^*}).
$$
\end{remark}

\begin{notation}\label{not:4.14}
Denote $v_k = \delta_{n_0-1,k}$ for every $k,\, 0\leq k<
2^{n_0-1}$. \end{notation}

\begin{notation}\label{not:4.15}
Denote $\mathcal{V} =\{ v_k,\, 0\leq k< 2^{n_0-1}\}$, where $v_k$
are defined in Notation~\ref{not:4.14}, and
$$ v_- =\min\limits_{v\in \mathcal{V}}\{ v,\, 1-v\}\text{ and }
v^+ =\max\limits_{v\in \mathcal{V}}\{ v,\, 1-v\}.
$$
\end{notation}

The following remark follows from Remark~\ref{rem:4.13}

\begin{remark}\label{rem:4.16}
For any $n\geq n_0$ and any $p$ and $k$ such that $I_{n, k}\subset
I_{n_0-1,p}$ the restriction $$\# I_{n_0-1, p}\cdot
(v_-)^{n-n_0+1}\leq  \# I_{n,k} \leq \# I_{n_0-1, p}\cdot
(v^+)^{n-n_0+1}
$$ holds.
\end{remark}

\begin{remark}\label{rem:4.17}
Let $n\geq 1$ and $k,\, 0\leq k< 2^n$ has binomial
expansion~\eqref{eq:4.1}. Then
$$ g(I_{n+1,k}) = I_{n,p},
$$ where the binomial expansion of $p$ is $$
p = \left\{ \begin{array}{ll} x_2\ldots x_n & \text{if }\, x_1
=0,\\
\Rot(x_2)\ldots \Rot(x_n) & \text{if }\, x_1 =1.
\end{array}\right.
$$
\end{remark}

The next remark follows from Remarks~\ref{rem:4.13}
and~\ref{rem:4.17} by induction.

\begin{remark}\label{rem:4.18}
Let $n\geq n_0$ and $k,\, 0\leq k< 2^{n-1}$ has the binary
expansion~\eqref{eq:4.1}. Then $\delta_{n+1,k} =
\Rot^\alpha(\delta_{n_0,p}),$ where $p =
\sum\limits_{i=n-n_0+2}^{n}\Rot^\alpha(x_i)2^{n-i}$,
\begin{equation}\label{eq:4.5} \alpha_{n-n_0+1} =
\sum\limits_{i=1}^{n-n_0+1}|x_i-x_{i-1}|\end{equation} and
$x_0=0$.
\end{remark}

\begin{remark}\label{rem:4.19}
Let~\eqref{eq:4.1} be the binary expansion of a natural number $k$
and let $\alpha_n$ be defined by~\eqref{eq:4.5}. Then
$(-1)^{\alpha(k)}=(-1)^{x_n}$.
\end{remark}

\begin{proof}
For every $i<n$ such that $x_i=1$ denote $p$ the number, which is
obtained from $k$ by the change of $x_i$ to $0$. Then $$
\alpha_n(k)-\alpha_n(p) =(1-x_{i-1})+(1-x_{i+1})-x_{i-1}-x_{i+1} =
2(1-x_{i-1}-x_{i+1}),
$$ which is an even number, whence
$(-1)^{\alpha(k)}=(-1)^{\alpha(p)}$ and the necessary equality
follows.
\end{proof}

By Remark~\ref{rem:4.19}, we can rewrite Remark~\ref{rem:4.18} as

\begin{remark}\label{rem:4.20}
Let $n\geq n_0$ and $k,\, 0\leq k< 2^{n-1}$ has the binary
expansion~\eqref{eq:4.1}. Then $\delta_{n+1,k} =
\Rot^{x_{n-n_0+1}}(\delta_{n_0,p_n}),$ where
\begin{equation}\label{eq-4.6}p_{i} =
\sum\limits_{j=i-n_0+2}^{i}\Rot^{x_{i+1-n_0}}(x_j)2^{\, i-j}
\end{equation} for all $i,\, n_0+1\leq i\leq n$.
\end{remark}

\begin{lemma}\label{lema:4.22}
For any $n< n_0$ and $k,\, 0\leq k\leq
2^{n-1}$,\begin{equation}\label{eq-4.7}\# I_{n+1,k_n} = \#
I_{n_0,k_{n_0-1}}\cdot
\prod\limits_{i=n_0+1}^{n}\Rot^{x_i+x_{i-n_0+1}}(\delta_{n_0,p_{i-1}})
\end{equation} holds, where $p_{i-1}$ is defined by~\eqref{eq-4.6}
for all $i,\, n_0+1\leq i\leq n$.
\end{lemma}

\begin{proof}
By Remark~\ref{rem:4.10},
$$
\# I_{n+1,k_n} = \# I_{n,k_{n-1}}\cdot
\Rot^{x_n}(\delta_{n,k_{n-1}}).
$$ Now, by Remark~\ref{rem:4.20}, we can continue $$
\# I_{n,k_{n-1}}\cdot \Rot^{x_n}(\delta_{n,k_{n-1}}) =  \#
I_{n,k_{n-1}}\cdot \Rot^{x_n+x_{n-n_0+1}}(\delta_{n_0,p_{n-1}}),
$$ where $p_{n-1}$ is defined in~\eqref{eq-4.6}. %
Now~\eqref{eq-4.7} follows by induction.
\end{proof}

We are now ready to prove Theorem~\ref{th:2}.

\begin{proof}[Proof of Theorem~\ref{th:2}]
Denote $d = \max\limits_{k,\, 0\leq k< 2^{n_0-1}}\# I_{n_0-1,k}$.
By Remark~\ref{rem:4.16}, $$\#I_{n,k}\leq d\cdot (v^+)^{n-n_0+1}
$$ for any $n\geq n_0$ and $k,\, 0\leq k\leq 2^{n-1}$.

Since $v^+<1$, then $$\lim\limits_{n\rightarrow
\infty}\left(\max\limits_{k,\, 0\leq k< 2^{n-1}}\# I_{n,k}\right)
\leq \lim\limits_{n\rightarrow \infty} d\cdot (v^+)^{n-n_0+1} =0$$
and theorem follows.
\end{proof}

\section{The derivative of the conjugation}\label{sec:5}

We will start this section with some technical calculations, which
we will use for the further reasonings.

\subsection{Technical calculations}

\begin{notation}
Let $g$ be a firm carcass map, which will be fixed up to the end
of this section.

Let $x\in [0,\, 1]$ be fixed till the end of the section,
$\{x_n,\, n\geq 0\}$ be its $g$-expansion and $k_n$ be defined
by~\eqref{eq:4.1}. Denote $\widehat{x}_n =\mu_{n+1,k_n}, $
$\widehat{x}_n^{\, +} = \mu_{n+1,k_n+1},$ $\widehat{x}_n^{\, \pm}
= \mu_{n+2,2k_n+1}, $ $ \widehat{x}_n^{\, -} =\mu_{n+1,k_n-1} $
and $\widehat{x}_n^{\, \mp} =\mu_{n+2,2k_n-1}$.
\end{notation}

\begin{remark}\label{rem-01}
Notice, that, by Notation~\ref{not:4.7},\\ $I_{n+1,k_n} =
(\widehat{x}_n, \widehat{x}_n^{\, +})$ and\\
$I_{n+2,2k_n} = (\widehat{x}_n, \widehat{x}_n^{\, \pm})$.
\end{remark}

\begin{remark}
Notice that $$\widehat{x}_n^{\, -} \leq \widehat{x}_n^{\, \mp}\leq
x\leq \widehat{x}_n^{\, \pm}\leq \widehat{x}_n^{\, +}
$$ for every $n\geq 1$. Moreover, %
all ``$\leq$'' are ``$<$'', whenever $x\not\in \{0,\, 1\}$.
\end{remark}

Suppose that $x_{n+1}=1$. Then we can specify the
$g$-expansion~\eqref{eq:4.3} of $x$ as
\begin{equation}\label{eq:5.1}
x \stackrel{g}{\longleftrightarrow} 0x_1x_2\ldots x_{n}\, 1\,
\underbrace{0\ldots 0}_{t\text{ zeros}}\, 1\,
x_{n+t+3}x_{n+t+4}\ldots\, .
\end{equation}

\begin{remark}
If~\eqref{eq:5.1} is the $g$-expansion of $x\in [0,\, 1]$, then
$$x_{n+1} =x_{n+t+2}=1.$$
$$
x_{n+2}=x_{n+3} =\ldots =x_{n+t+1}=0.
$$
$$
\widehat{x}_{n+1} = \widehat{x}_{n+2} = \ldots =
\widehat{x}_{n+t+1}.
$$
\end{remark}

Denote \begin{equation}\label{eq:5.2} b_i =
\widehat{x}_{n+i+1}^{\, -},\, 0\leq n\leq t.
\end{equation}

\begin{calcul}\label{calc:5.4}
If $x\in (0,\, 1)$ has $g$-expansion~\eqref{eq:5.1} and $\{b_i,\,
i\geq 0\}$ is given by~\eqref{eq:5.2}, then $b_0 = \widehat{x}_n$
and
\begin{equation}\label{eq:5.3}
b_i = \widehat{x}_{n+i}^{\, \mp} \end{equation} %
for all $i,\, 1\leq i\leq t$.
\end{calcul}

\begin{proof}
Since $x_{n+1}=1$, then $\widehat{x}_{n+1}^{\, -} =
\widehat{x}_n$, whence $b_0 = \widehat{x}_n$. Next, $$
\widehat{x}_{n+1} = \widehat{x}_{n+2} = \ldots \widehat{x}_{n+t}
=\widehat{x}_{n+t+1} \neq \widehat{x}_{n+t+2}
$$ implies $\widehat{x}_{n+i+1}^{\, -} = \widehat{x}_{n+i}^{\, \mp}$ for
$i=1,\ldots,\, t$, whence~\eqref{eq:5.3} follows.
\end{proof}

\begin{lemma}\label{lema:5.5}
Suppose that $n\geq n_0-1$. Then the following hold:

1. For any $i,\, 0\leq i\leq n_0+1$ write
\begin{equation}\label{eq:5.4}
\# ([\widehat{x}_n,\, \widehat{x}_n^+])\cdot (v_-)^{i+1} \leq \#
([b_i,\, \widehat{x}_{n+1}]) = \# ([\widehat{x}_n,\,
\widehat{x}_n^+])\cdot (v^+)^{i+1}.
\end{equation}

2. For every $i,\, n_0+2\leq i\leq t$ we have that
\begin{equation}\label{eq:5.5}
\# ([\widehat{x}_n,\, \widehat{x}_n^+])\cdot (v_-)^{n_0+2}\cdot
v_0^{i-n_0-1} \leq \# ([b_i,\, \widehat{x}_{n+1}]) \leq \#
([\widehat{x}_n,\, \widehat{x}_n^+])\cdot (v^+)^{n_0+2}\cdot
v_0^{i-n_0-1}.
\end{equation}

3. For any $i,\, 1\leq i\leq n_0$ the restriction
\begin{equation}\label{eq:5.6}
\# [\widehat{x}_{n}, \widehat{x}_n^+]\cdot (v_-)^i\leq \#
[\widehat{x}_{n+1}, \widehat{x}_{n+i}^+]\leq \# [\widehat{x}_{n},
\widehat{x}_n^+]\cdot (v^+)^i
\end{equation} holds.

4. For every $i,\, n_0\leq i\leq t-n_0+1$ we have that
\begin{equation}\label{eq:5.7}
\# [\widehat{x}_{n+1}, \widehat{x}_{n+i}^+] = \#
[\widehat{x}_{n+1}, \widehat{x}_{n+n_0}^+]\cdot v_0^{i-n_0}
\end{equation} and
\begin{equation}\label{eq:5.8}
\# [\widehat{x}_{n}, \widehat{x}_n^+]\cdot (v_-)^{n_0}\cdot
v_0^{i-n_0}\leq \# [\widehat{x}_{n+1}, \widehat{x}_{n+i}^+] \leq
\# [\widehat{x}_{n}, \widehat{x}_n^+]\cdot (v^+)^{n_0}\cdot
v_0^{i-n_0}.
\end{equation}

5. For every $i,\, t-n_0+2\leq i\leq t$ we have
\begin{equation}\label{eq:5.9}
\begin{array}{l}
\# [\widehat{x}_{n}, \widehat{x}_n^+]\cdot v_0^{t-2n_0+1}\cdot
(v_-)^{i-t+2n_0-1} \leq \# [\widehat{x}_{n+1},
\widehat{x}_{n+i}^+] \leq\\
\# [\widehat{x}_{n}, \widehat{x}_n^+] \cdot v_0^{t-2n_0+1}\cdot
(v^+)^{i-t+2n_0-1}
\end{array}
\end{equation}
\end{lemma}

\begin{proof}
The restriction~\eqref{eq:5.4} follows from Remark~\ref{rem:4.10}.

By Calculation~\ref{calc:5.4} the $g$-expansion of
$\widehat{x}_{n+1}$ is $$ \widehat{x}_{n+1}
\stackrel{g}{\longleftrightarrow} x_0x_1\ldots x_n1,$$ the
$g$-expansion of $b_i$ is $$b_i \stackrel{g}{\longleftrightarrow}
x_0x_1\ldots x_n0\underbrace{1\ldots 1}_{i\text{ ones}}$$ %
and $b_i = \mu_{n+i+1,p_i}$, where the binary expansion of $p_i$
is $$ p_i  = x_1\ldots x_n0\underbrace{1\ldots 1}_{i\text{
ones}}\,.$$ Thus, $$ (b_i, \widehat{x}_{n+1}) = I_{n+2+i,p_i}.
$$

By Remark~\ref{rem:4.20} for any $i,\, 0\leq i\leq t$ denote $q_i$
the number with the binary expansion $$ q_i =
\Rot^{x_{n+3+i-n_0}}(x_{n+4+i-n_0})\ldots\,
\Rot^{x_{n+3+i-n_0}}(x_{n+1+i})
$$ and
write $$ \delta_{n+2+i,p_i} =
\Rot^{x_{n+3+i-n_0}}%
(\delta_{n_0-1,\ q_i}).
$$
If $n_0< i\leq t,$ then
\begin{equation}\label{eq:5.10}
\delta_{n+2+i,p_i} =
\Rot^{1}%
\left(\delta_{n_0-1, \Rot^1(q_i)}\right) = 1 - \delta_{n_0-1,0} =
1-v_0.
\end{equation}

If $i\geq n_0+2$, then~\eqref{eq:5.10} implies $$ \# ([b_i,\,
\widehat{x}_{n+1}]) = \# ([b_{i-1},\, \widehat{x}_{n+1}])\cdot
\Rot^1(\delta_{n+1+i,\, p_{i-1}}) = \# ([b_{i-1},\,
\widehat{x}_{n+1}])\cdot v_0,
$$ whence $$
\# ([b_i,\, \widehat{x}_{n+1}]) = \# ([b_{n_0+1},\,
\widehat{x}_{n+1}])\cdot v_0^{i-n_0-1}
$$
and~\eqref{eq:5.5} follows.

The restriction~\eqref{eq:5.6} follows from Remark~\ref{rem:4.16}.

Remind that $[\widehat{x}_{n+1}, \widehat{x}_{n+i}^+] =
[\widehat{x}_{n+i}, \widehat{x}_{n+i}^+] = I_{n+i+1, k_{n+i}}$.

Notice, that for every $i,\, 1\leq i<t$ the last $i-1$ binary
digits of $k_{n+i}$ are zero and the last binary digit of
$k_{n+t}$ is $1$. Thus, if $n_0\leq i\leq t-n_0+1$, then the last
$n_0-1$ binary digits of $k_{n+i}$ are zero and, by
Remark~\ref{rem:4.20},
$$
\delta_{n+i+1,k_{n+i}} = \delta_{n_0-1,0}.
$$ %
Now, by Remark~\ref{rem:4.13}, $$ \# [\widehat{x}_{n+1},
\widehat{x}_{n+i}^+] = \# [\widehat{x}_{n+1},
\widehat{x}_{n+i-1}^+]\cdot v_0,
$$ whence~\eqref{eq:5.7} follows by induction.

Restrictions~\eqref{eq:5.8} and~\eqref{eq:5.9} follow
from~\eqref{eq:5.7} by Remark~\ref{rem:4.16}.
\end{proof}

\subsection{The case of firm carcass maps}

Let $g_1$ and $g_2$ be firm carcass maps, which will be fixed to
the end of the section. Denote by $n_0$ the minimal natural number
such that $g_i^{-n_0}(0)$ contains all the kinks of $g_i$, where
$i\in \{1,\, 2\}$. For every $n\geq 1$ evert $k,\, 0\leq k\leq
2^{n-1}$ and $i\in \{1;\, 2\}$ define $\mu_{n,k}(g_i)$,
$I_{n,k}(g_i)$, $\delta_{n,k}(g_i)$, $v(g_i)$, $\mathcal{V}(g_i)$,
$v_-(g_i)$ and $v^{+}(g_i)$ as in Notations~\ref{not:4.7},
\ref{not:4.9} and~\ref{not:4.15}.

\begin{lemma}\label{lema:5.6}
For every $x\in [0,\, 1]$ the $g_2$-expansion of $h(x)$ coincides
with the $g_1$-expansion of $x$.
\end{lemma}

\begin{notation}
Denote by $h$ be the unique continuous invertible solution of the
functional equation~\eqref{eq:1.5}. Let a number $x\in [0,\, 1]$
be fixed till the end of the section and let~\eqref{eq:4.3} be its
$g_1$-expansion.
\end{notation}

\begin{remark}\label{rem:5.8}
If $x\notin g_1^{-n}(0)$ for some $n\geq 1$ then
$$
h_{n+1}'(x) =\frac{h(\widehat{x}_n^{\, +})
-h(\widehat{x}_n)}{\widehat{x}_n^{\, +} -\widehat{x}_n}.
$$
\end{remark}

\begin{remark}
If $x\notin g_1^{-n}(0)$ for some $n> 1$ then there exist $v\in
\mathcal{V}(g_1)$ and $w\in \mathcal{V}(g_2)$ such that
$$
h_n'(x) = h_{n-1}'(x)\cdot \frac{v}{w}.
$$
\end{remark}

\begin{remark}
If for some $x\in [0,\, 1]$ the derivative $h'(x)$ exists, then

1. If $x>0$, then there exists
$\lim\limits_{n\rightarrow\infty}h_n'(x-)$, and $h'(x)
=\lim\limits_{n\rightarrow\infty}h_n'(x-)$;

2. If $x<1$, then there exists
$\lim\limits_{n\rightarrow\infty}h_n'(x+)$, and $h'(x)
=\lim\limits_{n\rightarrow\infty}h_n'(x+)$.
\end{remark}

Suppose that for some $n\geq 1$ the $g_1$-expansion~\eqref{eq:4.3}
can be specified as~\eqref{eq:5.1}. The next fact follows
from~\eqref{eq:5.7}, Lemma~\ref{lema:5.6} and
Remark~\ref{rem:5.8}.

\begin{remark}\label{rem:5.11}
For every $i\geq 1$ $$ \left( \frac{v_-(g_2)}{v^+(g_1)}\right)^{i}
\cdot h_{n+1}'(x)\leq h_{n+i}'(x)\leq \left(
\frac{v^+(g_2)}{v_-(g_1)}\right)^{i} \cdot h_{n+1}'(x)
$$
\end{remark}

\begin{remark}\label{rem:5.12}
For every $i,\, n_0\leq i\leq t-n_0+1$

$$
h_{n+i}'(x) = h_{n_0+n}'(x)\cdot \left(
\frac{v_0(g_2)}{v_0(g_1)}\right)^{i-n_0}
$$
\end{remark}

\begin{lemma}\label{lema:5.13}
Suppose that $g_1^{n+1}(x)=0$ and $ s\in [\widehat{x}_n^{\, -},\,
\widehat{x}_n^{\, \mp})$ for $n>n_0$. Then
$$ v_-(g_2)\cdot \widehat{h}_{n+1}'(\widehat{x}_n-)\leq
\frac{h(\widehat{x}_n)-h(s)}{\widehat{x}_n-s}\leq
\widehat{h}_{n+1}'(\widehat{x}_n-)\cdot \frac{1}{v_(g_1)}.
$$
\end{lemma}

\begin{proof}
Since $g_1^{n+1}(x)=0$, then $x = \widehat{x}_n$. Denote
$A(\widehat{x}_n^{\, -}, h(\widehat{x}_n^{\, -}))$, $S(s, h(s))$,
$X(x,\, h(x))$, $S_-(\widehat{x}_n^{\, -}, h(\widehat{x}_n^{\,
\mp}))$ and $S^+(\widehat{x}_n^{\, \mp}, h(\widehat{x}_n^{\, -}))$
(see Fig.~\ref{fig-1}a).

Also let $k_{S_-X}$ be the tangent of $S_-X$, let $k_{SX}$ be
$k_{SX}$ and let $k_{S^+X}$ be the tangent  of $S^+X$. Then
 $$ k_{S_-X}\leq k_{SX}\leq k_{S^+X},
$$ because
$ s\in [\widehat{x}_n^{\, -},\, \widehat{x}_n^{\, \mp})$ and $h$
increase (by Remark~\ref{rem:3.7}).

By Remark~\ref{rem:4.16} and Lemma~\ref{lema:5.6}
$$
k_{S_-X}\geq \frac{(h(x)-h(\widehat{x}_n^{\, -}))\cdot
v_-(g_2)}{x- \widehat{x}_n^{\, -}}
$$ and $$
k_{S^+X} \leq \frac{h(x)-h(\widehat{x}_n^{\, -})}{(x -x_n^{\,
-})\cdot v_-(g_1)}.
$$ %
Now lemma follows from Remark~\ref{rem:5.8}.
\end{proof}

\begin{lemma}\label{lema:5.14}
Suppose that $x$ is $g_1$-infinite and $ s\in [\widehat{x}_n,\,
\widehat{x}_n^{\, +})$ for $n>n_0$. There exist $k_-$ and $k^+$,
independent on $x$, and $i\geq 1$ such that
$$
k_-\cdot h_{n+i}'(\widehat{x}_n-)  \leq
\frac{h(\widehat{x}_n)-h(s)}{\widehat{x}_n-s} \leq k^+\cdot
h_{n+i}'(\widehat{x}_n-)\ .
$$
\end{lemma}

\begin{proof}

Define $\{b_i,\, 0\leq i\leq t\}$ by~\eqref{eq:5.2}.

Notice, that for any $s_-,\, s^+,\, x_-$ and $x^+$ such that $$
\left\{ \begin{array}{l} s_-\leq s\leq s^+<x^-\leq x\leq x^+\\
s_-<s^+<x_-<x^+
\end{array}\right.
$$ if follows from Remark~\ref{rem:3.7} that
$$
k_{S_-X^+}\leq k_{SX}\leq k_{S^+X_-}, $$ %
where:

1. $k_{SX}$ is the tangent of the line, which connects points $(s,
h(s))$ and $(x, h(x))$;

2. $k_{S_-X^+}$ is the tangent of the line, which connects points
$(s_-, h(s^+))$ and $(x^+, h(x_-))$ and, finally

3. $k_{S^+X_-}$ is the tangent of the line, which connects points
$(s^+, h(s_-))$ and $(x_-, h(x^+))$.

By definitions, $$ k_{SX} =\frac{h(x)-h(s)}{x-s}, $$
$$
k_{S_-X^+} = \frac{h(x_-)-h(s^+)}{x^+-s_-}, $$ and
$$
k_{S^+X_-} = \frac{h(x^+)-h(s_-)}{x_--s^+}.
$$

\begin{figure}[htbp]
\begin{minipage}[h]{0.45\linewidth}
\begin{center}
\begin{picture}(110,110)
\put(0,0){\line(0,1){96}} \put(0,0){\line(1,0){96}}

\put(0,48){\line(1,0){96}} 
\put(0,96){\line(1,0){96}}

\put(48,0){\line(0,1){96}} 
\put(96,0){\line(0,1){96}}

\put(0,0){\circle*{5}} \put(96,96){\circle*{5}} \put(-17,0){$A$}
\put(100,100){$X$}


\put(48,48){\circle*{3}} 


\put(21,24){\circle*{3}} \put(15,11){$S$}

\put(48,0){\circle*{3}} \put(35,5){$S^+$}

\put(0,48){\circle*{3}} \put(5,35){$S_-$}

\qbezier(0,48)(48,72)(96,96) \qbezier(48,0)(72,48)(96,96)
\end{picture}
\vskip 3mm \centerline{a) Construction of} \centerline{$S$, $S^+$,
$S_-$ and $X$}\end{center}
\end{minipage}
\hfill
\begin{minipage}[h]{0.45\linewidth}
\begin{center}
\begin{picture}(110,110)
\put(0,0){\line(0,1){96}} \put(0,0){\line(1,0){96}}

\put(0,48){\line(1,0){96}} \put(48,72){\line(1,0){48}}
\put(0,96){\line(1,0){96}}

\put(48,0){\line(0,1){96}} \put(72,48){\line(0,1){48}}
\put(96,0){\line(0,1){96}}

\put(0,0){\circle*{5}} \put(96,96){\circle*{5}} \put(-17,0){$A$}
\put(100,100){$B$}

\put(80,80){\circle*{3}} \put(80,85){$X$}

\put(48,48){\circle*{3}} \put(37,37){$C$}

\put(21,34){\circle*{3}} \put(15,21){$S$}

\qbezier(0,48)(48,60)(96,72) \qbezier(48,0)(60,48)(72,96)

\put(0,48){\circle*{3}} \put(48,0){\circle*{3}}
\put(72,96){\circle*{3}} \put(96,72){\circle*{3}}

\put(2,55){$S_-$}%
\put(50,3){$S^+$}

\put(53,85){$X_-$}%
\put(78,57){$X^+$}

\end{picture}
\vskip 3mm \centerline{b) Case $t =0$}\centerline{in
Lemma~\ref{lema:5.14}}\end{center}
\end{minipage}
\hfill \caption{Proof of Lemmas~\ref{lema:5.13}
and~\ref{lema:5.14}} \label{fig-1}
\end{figure}
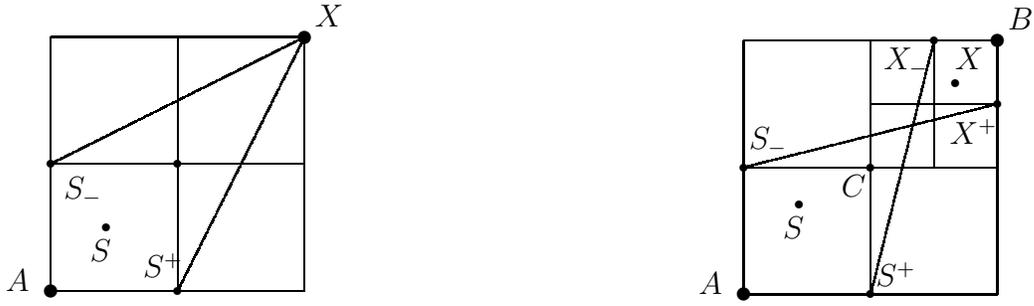

If $s\in [b_i, b_{i+1})$ for some $i, 0\leq i< t$ then take
\begin{equation}\label{eq:5.11}
\begin{array}{ll}
s_- = b_i,& s^+ =b_{i+1},\\
x_- = \widehat{x}_{n+1},& x^+ =\widehat{x}_{n+i+2}^{\, +}.
\end{array}
\end{equation}

And if $s\in [b_t, \widehat{x}_{n+1})$ then take
\begin{equation}\label{eq:5.12}
\begin{array}{ll}
s_- = b_t,& s^+ =\widehat{x}_{n+1},\\
x_- = \widehat{x}_{n+t+2},& x^+ =\widehat{x}_{n+t+1}^{\, +}.
\end{array}
\end{equation}

The case $t=0$ is presented at Figure~\ref{fig-1}b.

\begin{figure}[htbp]
\begin{minipage}[h]{0.45\linewidth}
\begin{center}
\begin{picture}(110,110)
\put(0,0){\line(0,1){96}} \put(0,0){\line(1,0){96}}

\put(0,48){\line(1,0){96}} \put(48,72){\line(1,0){48}}
\put(0,96){\line(1,0){96}}

\put(48,0){\line(0,1){96}} \put(72,48){\line(0,1){48}}
\put(96,0){\line(0,1){96}}

\put(0,0){\circle*{5}} \put(96,96){\circle*{5}} \put(-17,0){$A$}
\put(100,100){$B$}

\qbezier(24,0)(24,24)(24,48) \qbezier(0,24)(24,24)(48,24)

\put(60,60){\circle*{3}} \put(62,55){$X$}


\put(15,15){\circle*{3}} \put(10,3){$S$}


\qbezier(0,24)(36,36)(72,48) \qbezier(24,0)(36,36)(48,72)
\put(0,24){\circle*{3}} \put(24,0){\circle*{3}}
\put(48,72){\circle*{3}} \put(72,48){\circle*{3}}
\end{picture}
\vskip 3mm \centerline{a) $s\in [\widehat{x}_n,
\widehat{x}_{n+1}^{\, \mp})$}\end{center}
\end{minipage}
\hfill
\begin{minipage}[h]{0.45\linewidth}
\begin{center}
\begin{picture}(110,110)
\put(0,0){\line(0,1){96}} \put(0,0){\line(1,0){96}}

\put(0,48){\line(1,0){96}} \put(48,72){\line(1,0){48}}
\put(0,96){\line(1,0){96}}

\put(48,0){\line(0,1){96}} \put(72,48){\line(0,1){48}}
\put(96,0){\line(0,1){96}}

\put(0,0){\circle*{5}} \put(96,96){\circle*{5}}

\qbezier(24,0)(24,24)(24,48) \qbezier(0,24)(24,24)(48,24)

\put(66,66){\circle*{3}} 


\put(36,36){\circle*{3}} 

\qbezier(60,48)(60,60)(60,72) \qbezier(48,60)(60,60)(72,60)

\qbezier(24,48)(54,54)(72,60) \qbezier(48,24)(54,54)(60,72)
\put(24,48){\circle*{3}} \put(48,24){\circle*{3}}
\put(72,60){\circle*{3}} \put(60,72){\circle*{3}}
\end{picture}
\vskip 3mm \centerline{b) $s\in [\widehat{x}_{n+1}^{\, \mp},
\widehat{x}_{n+1})$}\end{center}
\end{minipage}
\hfill \caption{Case $t=1$ in Lemma~\ref{lema:5.13}} \label{fig-2}
\end{figure}
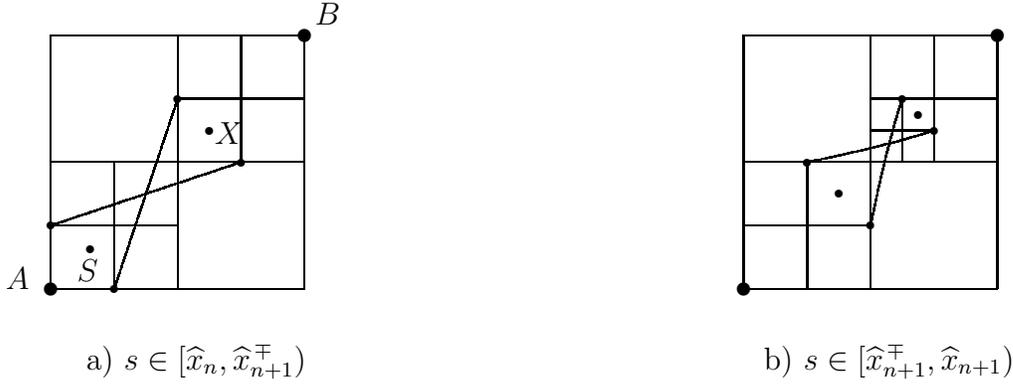

The case $t=1$ is illustrated at Figure~\ref{fig-2}.

Suppose that $s\in [b_i, b_{i+1})$ for some $i,\, 0\leq i<t$, and
that the numbers $s_-,\, s^+,\, x_-$ and $x^+$ are defined
by~\eqref{eq:5.11}.

Since $ x_- -s^+ = \widehat{x}_{n+1} - b_{i+1}$ we can use
Lemma~\ref{lema:5.5} directly to obtain the restrictions for $x_-
-s^+$. In order to use the same Lemma for the restriction of $x^+
-s_- =\widehat{x}^+_{n+i+2} -b_i$, we will write $$ x^+ -s_- =
(\widehat{x}^+_{n+i+2}- \widehat{x}_{n+1}) + (\widehat{x}_{n+1}
-b_i)$$ and then use Lemma~\ref{lema:5.5} to obtain the
restriction of the each summand.

For $0\leq i\leq n_0$ by~\eqref{eq:5.4} of Lemma~\ref{lema:5.5},
\begin{equation}\label{eq:5.13}
x_- -s^+ \leq \# [\widehat{x}_n,\, \widehat{x}_n^+]
\end{equation} and
\begin{equation}\label{eq:5.14}
\# [\widehat{x}_n,\, \widehat{x}_n^+]\cdot (v_-(g_1))^{n_0+1} \leq
\widehat{x}_{n+1} -s_- \ .\end{equation}

If $n_0+1\leq i<t$, then by~\eqref{eq:5.5} of
Lemma~\ref{lema:5.5},
\begin{equation}\label{eq:5.15}
x_- -s^+ \leq \# [\widehat{x}_n,\, \widehat{x}_n^+] \cdot
(v^+(g_1))^{n_0+2}\cdot (v_0(g_1))^{i-n_0}
\end{equation} and
\begin{equation}\label{eq:5.16}
\# ([\widehat{x}_n,\, \widehat{x}_n^+])\cdot (v_-)^{n_0+2}\cdot
v_0^{i-n_0-1} \leq \widehat{x}_{n+1} -s_-
\end{equation}

For $i,\, 0\leq i\leq n_0-2$ it follows from~\eqref{eq:5.6} of
Lemma~\ref{lema:5.5} that

\begin{equation}\label{eq:5.17}
\# [\widehat{x}_{n}, \widehat{x}_n^+]\cdot (v_-(g_1))^{n_0}\leq
x^+ -\widehat{x}_{n+1} \ .\end{equation}

If $n_0-1\leq i\leq t-n_0-1$ then it follows from~\eqref{eq:5.8}
of Lemma~\ref{lema:5.5} that
\begin{equation}\label{eq:5.18}
\# [\widehat{x}_{n}, \widehat{x}_n^+]\cdot (v_-(g_1))^{n_0}\cdot
(v_0(g_1))^{i+2-n_0}\leq x^+ -\widehat{x}_{n+1} \ .\end{equation}

If $t-n_0\leq i<t$, then it follows from~\eqref{eq:5.9} of
Lemma~\ref{lema:5.5} that \begin{equation}\label{eq:5.19} \#
[\widehat{x}_{n}, \widehat{x}_n^+]\cdot v_0^{t-2n_0+1}\cdot
(v_-)^{2n_0} \ \ \ \leq x^+ -\widehat{x}_{n+1}
\end{equation}

If $0\leq i\leq n_0-2$ then if follows from~\eqref{eq:5.14}
and~\eqref{eq:5.17} then \begin{equation}\label{eq:5.20}\#
[\widehat{x}_{n}, \widehat{x}_n^+]\cdot 2\cdot (v_-)^{n_0+1}\leq
x^+ -s_-
\end{equation}

If $n_0-1\leq i\leq n_0$ then it follows from~\eqref{eq:5.14}
and~\eqref{eq:5.18} that \begin{equation}\label{eq:5.21}\#
[\widehat{x}_{n}, \widehat{x}_n^+]\cdot 2\cdot (v_-)^{n_0+1}\cdot
v_0^{i+2-n_0}\leq x^+ -s_-
\end{equation}

If $n_0+1\leq i\leq t-n_0-1$ then, by~\eqref{eq:5.16}
and~\eqref{eq:5.18} obtain
\begin{equation}\label{eq:5.22}\#
[\widehat{x}_{n}, \widehat{x}_n^+]\cdot 2\cdot (v_-)^{n_0+2}\cdot
v_0^{i+2-n_0}\leq x^+ -s_-
\end{equation}

If $t-n_0 \leq i\leq t-1$ then, by~\eqref{eq:5.16}
and~\eqref{eq:5.19} obtain \begin{equation}\label{eq:5.23} \#
[\widehat{x}_{n}, \widehat{x}_n^+]\cdot 2 \cdot (v_-)^{2n_0} \cdot
v_0^{t-n_0} \leq x^+ -s_- \ .
\end{equation}

$$ K_{S_-X^+} \geq \frac{ \min(x^+ -s_-)}{\max(x_- -s^+)}
$$ and $$
K_{S^+X_-} \leq \frac{\max(x_- -s^+)}{\min(x^+ -s_-)}
$$

If $0\leq i\leq n_0-2$ then, by~\eqref{eq:5.13}
and~\eqref{eq:5.20}

\begin{equation}\label{eq:5.24}
K_{S_-X^+} \geq \frac{\# [\widehat{y}_n,\, \widehat{y}_n^+]\cdot
2\cdot (v_-(g_2))^{n_0+1}}{\# [\widehat{x}_{n}, \widehat{x}_n^+]}
\end{equation} and \begin{equation}\label{eq:5.25} K_{S^+X_-} \leq \frac{\# [\widehat{y}_n,\,
\widehat{y}_n^+]}{\# [\widehat{x}_{n}, \widehat{x}_n^+]\cdot
2\cdot (v_-(g_1))^{n_0+1}}
\end{equation}

If $n_0-1\leq i\leq n_0$ then by~\eqref{eq:5.13}
and~\eqref{eq:5.21}
\begin{equation}\label{eq:5.26}
K_{S_-X^+} \geq \frac{\# [\widehat{y}_n,\, \widehat{y}_n^+]\cdot
2\cdot (v_-(g_2))^{n_0+1}\cdot (v_0(g_2))^{i+2-n_0}}{\#
[\widehat{x}_{n}, \widehat{x}_n^+]}
\end{equation} and \begin{equation}\label{eq:5.27}
K_{S^+X_-} \leq \frac{\# [\widehat{y}_n,\, \widehat{y}_n^+]}{\#
[\widehat{x}_{n}, \widehat{x}_n^+]\cdot 2\cdot
(v_-(g_1))^{n_0+1}\cdot (v_0(g_1))^{i+2-n_0}}
\end{equation}

If $n_0+1\leq i\leq t-n_0-1$ then by~\eqref{eq:5.15}
and~\eqref{eq:5.22} \begin{equation}\label{eq:5.28} K_{S_-X^+}
\geq \frac{\# [\widehat{y}_n,\, \widehat{y}_n^+]\cdot 2\cdot
(v_-(g_2))^{n_0+2}\cdot (v_0(g_2))^{i+2-n_0}}{\# [\widehat{x}_{n},
\widehat{x}_n^+]\cdot (v^+(g_1))^{n_0+2}\cdot (v_0(g_1))^{i-n_0}}
\end{equation} and \begin{equation}\label{eq:5.29}
K_{S^+X_-} \leq \frac{\# [\widehat{y}_n,\, \widehat{y}_n^+]\cdot
(v^+(g_2))^{n_0+2}\cdot (v_0(g_2))^{i-n_0}}{\# [\widehat{x}_{n},
\widehat{x}_n^+]\cdot 2\cdot (v_-(g_1))^{n_0+2}\cdot
(v_0(g_1))^{i+2-n_0}}
\end{equation}

If $t-n_0 \leq i\leq t-1$ then by~\eqref{eq:5.15}
and~\eqref{eq:5.23}
\begin{equation}\label{eq:5.30} K_{S_-X^+} \geq \frac{\# [\widehat{y}_n,\,
\widehat{y}_n^+]\cdot 2 \cdot (v_-(g_2))^{2n_0} \cdot
(v_0(g_2))^{t-n_0}}{\# [\widehat{x}_{n}, \widehat{x}_n^+]\cdot
(v^+(g_1))^{n_0+2}\cdot (v_0(g_1))^{i-n_0}}
\end{equation} and \begin{equation}\label{eq:5.31}
K_{S^+X_-} \leq \frac{\# [\widehat{y}_n,\, \widehat{y}_n^+]\cdot
(v^+(g_2))^{n_0+2}\cdot (v_0(g_2))^{i-n_0}}{\# [\widehat{x}_{n},
\widehat{x}_n^+]\cdot 2 \cdot (v_-(g_1))^{2n_0} \cdot
(v_0(g_1))^{t-n_0}} \end{equation}

Due to Remark~\ref{rem:5.8}, it follows from~\eqref{eq:5.24},
\eqref{eq:5.25}, \eqref{eq:5.26} and~\eqref{eq:5.27} that for any
$i,\, 0\leq i\leq n_0$
$$
K_{S_-X^+}\geq h_{n+1}'(x)\cdot 2\cdot (v_-(g_2))^{n_0+1}\cdot
(v_0(g_2))^2
$$
and
$$
K_{S^+X_-}\leq \frac{h_{n+1}'(x)}{ 2\cdot (v_-(g_1))^{n_0+1}\cdot
(v_0(g_1))^2}.
$$

By Remark~\ref{rem:5.11}, \begin{equation}\label{eq:5.32}\left(
\frac{v_-(g_2)}{v^+(g_1)}\right)^{n_0} \cdot h_{n+1}'(x)\leq
h_{n+n_0}'(x)\leq \left( \frac{v^+(g_2)}{v_-(g_1)}\right)^{n_0}
\cdot h_{n+1}'(x).\end{equation}

If $n_0+1\leq i\leq t-n_0-1$, then it follows from~\eqref{eq:5.28}
and~\eqref{eq:5.29}, Remarks~\ref{rem:5.8} and~\ref{rem:5.12}, and
from just obtained~\eqref{eq:5.32} that
$$ K_{S_-X^+} \geq h_{n+1}'(x)\cdot
\frac{2\cdot (v_-(g_2))^{n_0+2}\cdot
(v_0(g_2))^{i+2-n_0}}{(v^+(g_1))^{n_0+2}\cdot
(v_0(g_1))^{i-n_0}}\geq
$$
$$
\geq h_{n+1}'(x)\cdot \left(
\frac{v^+(g_2)}{v_-(g_1)}\right)^{n_0}\cdot \frac{2\cdot
(v_-(g_2))^{n_0+2}\cdot
(v_0(g_2))^{i+2-n_0}}{(v^+(g_1))^{n_0+2}\cdot
(v_0(g_1))^{i-n_0}}\cdot \left(
\frac{v_-(g_1)}{v^+(g_2)}\right)^{n_0}\geq
$$$$
\geq h'_{n_0+n}(x)\cdot \left(
\frac{v_0(g_2)}{v_0(g_1)}\right)^{i-n_0}\cdot \frac{2\cdot
(v_-(g_2))^2\cdot (v_0(g_2))^2}{(v^+(g_1))^2} \cdot \left(
\frac{v_-(g_1)\cdot v_-(g_2)}{v^+(g_1)\cdot
v^+(g_2)}\right)^{n_0}\geq
$$$$
\geq h'_{n+i}(x)\cdot \frac{2\cdot (v_-(g_2))^2\cdot
(v_0(g_2))^2}{(v^+(g_1))^2} \cdot \left( \frac{v_-(g_1)\cdot
v_-(g_2)}{v^+(g_1)\cdot v^+(g_2)}\right)^{n_0}
$$ and $$K_{S^+X_-} \leq h_{n+1}'(x)\cdot \frac{ (v^+(g_2))^{n_0+2}\cdot
(v_0(g_2))^{i-n_0}}{ 2\cdot (v_-(g_1))^{n_0+2}\cdot
(v_0(g_1))^{i+2-n_0}}\leq
$$$$ \leq h_{n+1}'(x)\cdot \left( \frac{v_-(g_2)}{v^+(g_1)}\right)^{n_0}
\cdot \frac{ (v^+(g_2))^{n_0+2}\cdot (v_0(g_2))^{i-n_0}}{ 2\cdot
(v_-(g_1))^{n_0+2}\cdot (v_0(g_1))^{i+2-n_0}}\cdot \left(
\frac{v^+(g_1)}{v_-(g_2)}\right)^{n_0}\leq
$$
$$ \leq h_{n_0+n}'(x)\cdot \left( \frac{v_0(g_2)}{v_0(g_1)}\right)^{i-n_0}\cdot
\frac{ (v^+(g_2))^{n_0+2}}{ 2\cdot (v_-(g_1))^{n_0+2}\cdot
(v_0(g_1))^2}\cdot \left(
\frac{v^+(g_1)}{v_-(g_2)}\right)^{n_0}\leq
$$$$
\leq h'_{n+i}(x)\cdot \frac{ (v^+(g_2))^2}{2\cdot
(v_0(g_1))^2\cdot (v_-(g_2))^2} \cdot \left( \frac{v^+(g_1)\cdot
v^+(g_2)}{v_-(g_1)\cdot v_-(g_2)}\right)^{n_0}.
$$

If $t-n_0\leq i\leq t-1$, then it follows from~\eqref{eq:5.30}
\eqref{eq:5.31}, \eqref{eq:5.32} and Remarks~\ref{rem:5.8}
and~\ref{rem:5.12} that
$$ K_{S_-X^+} \geq h_{n+1}'(x)\cdot \frac{
2 \cdot (v_-(g_2))^{2n_0} \cdot (v_0(g_2))^{t-n_0}}{
(v^+(g_1))^{n_0+2}\cdot (v_0(g_1))^{i-n_0}}\geq $$
$$ \geq h_{n_0+n}'(x)\cdot \frac{
2 \cdot (v_-(g_2))^{2n_0} \cdot (v_0(g_2))^{t-n_0}\cdot
(v_-(g_1))^{n_0}}{ (v^+(g_1))^{n_0+2}\cdot (v_0(g_1))^{i-n_0}
\cdot (v^+(g_2))^{n_0}}\geq
$$
$$
\geq h_{n+i}'(x)\cdot \frac{ 2 \cdot (v_-(g_2))^{2n_0} \cdot
(v_0(g_2))^{t-i}\cdot (v_-(g_1))^{n_0}}{ (v^+(g_1))^{n_0+2}\cdot
(v^+(g_2))^{n_0}}\geq
$$$$
\geq h_{n+i}'(x)\cdot \frac{ 2 \cdot (v_-(g_2))^{2n_0} \cdot
(v_0(g_2))^{n_0}\cdot (v_-(g_1))^{n_0}}{ (v^+(g_1))^{n_0+2}\cdot
(v^+(g_2))^{n_0}}=
$$$$
= h_{n+i}'(x)\cdot \frac{ 2 \cdot (v_-(g_2))^{n_0} }{
(v^+(g_1))^2}\cdot \left(  \frac{v_-(g_1)\cdot v_-(g_2)\cdot
v_0(g_2)}{v^+(g_1)\cdot v^+(g_2)}\right)^{n_0}
$$
and
$$ K_{S^+X_-}
\leq h_{n+1}'(x)\cdot\frac{ (v^+(g_2))^{n_0+2}\cdot
(v_0(g_2))^{i-n_0}}{ 2 \cdot (v_-(g_1))^{2n_0} \cdot
(v_0(g_1))^{t-n_0}}\leq $$
$$ \leq h_{n_0+n}'(x)\cdot\frac{ (v^+(g_2))^{n_0+2}\cdot
(v_0(g_2))^{i-n_0}\cdot (v^+(g_1))^{n_0}}{ 2 \cdot
(v_-(g_1))^{2n_0} \cdot (v_0(g_1))^{t-n_0}\cdot
(v_-(g_2))^{n_0}}\leq $$$$ \leq h_{n+i}'(x)\cdot\frac{
(v^+(g_2))^{n_0+2}\cdot (v^+(g_1))^{n_0}}{ 2 \cdot
(v_-(g_1))^{2n_0} \cdot (v_0(g_1))^{t-i}\cdot
(v_-(g_2))^{n_0}}\leq $$
$$ \leq h_{n+i}'(x)\cdot\frac{
(v^+(g_2))^{n_0+2}\cdot (v^+(g_1))^{n_0}}{ 2 \cdot
(v_-(g_1))^{2n_0} \cdot (v_0(g_1))^{n_0}\cdot (v_-(g_2))^{n_0}} =
$$$$
=h_{n+i}'(x)\cdot\frac{ (v^+(g_2))^2}{ 2 \cdot
(v_-(g_1))^{n_0}}\cdot \left( \frac{ v^+(g_2)\cdot v^+(g_1)}{
v_-(g_1) \cdot v_-(g_2)\cdot v_0(g_1)}\right)^{n_0}.
$$ This proves lemma for the case~\eqref{eq:5.11}.

The case~\eqref{eq:5.12} can be considered analogously.
\end{proof}

In the same manner as in Lemmas~\ref{lema:5.13}
and~\ref{lema:5.14} we can prove the following lemma.

\begin{lemma}\label{lema:5.15}
Suppose that $x<1$. There exist $k_-$ and $k^+$, independent
on~$x$, and $i\geq 1$ such that for any $n\geq n_0$ and $ s\in
(\widehat{x}_n,\, \widehat{x}_n^{\, +}]$ there exits $i\geq 1$
such that
$$
k_-\cdot h_{n+i}'(\widehat{x}_n+)  \leq
\frac{h(\widehat{x}_n)-h(s)}{\widehat{x}_n-s} \leq k^+\cdot
h_{n+i}'(\widehat{x}_n+)\ .
$$
\end{lemma}

Suppose that $x$ is $g$-finite and $\widetilde{n}$ is such that
$g_1^{\widetilde{n}}(x)=0$.

Denote $I_{n,k} = (\mu_{n,k}(g_1),\, \mu_{n,k+1}(g_1))$. For any
$t\in I_{n,k}$ denote $$ \Delta_L(I_{n,k}, t) = \left|
\frac{h(t)-\mu_{n,k}(g_2)}{h_n'(t)\cdot
(t-\mu_{n,k}(g_1))}-1\right|
$$ and $$ \Delta_R(I_{n,k}, t) = \left|
\frac{\mu_{n,k+1}(g_2)-h(t)}{h_n'(t)\cdot
(\mu_{n,k+1}-t)}-1\right|\ .
$$

Notice that the following conditions are equivalent:

1. $\Delta_L(I_{n,k}, t) =0$;

2. $\Delta_R(I_{n,k}, t) =0$;

3. The point $(t, h(t))$ belongs to the graph of $h_n$.

The next lemma follows from Lemma~\ref{lema:2.8}.

\begin{lemma}\label{lema:5.16}
For every $n>n_0$, every $k,\, 0\leq k< 2^{n-1}$ and $t\in
I_{n,k}$ we have that

1. If $g_1$ increase on $I_{n,k}$, then $$ \Delta_L(I_{n,k}, t) =
\Delta_L(g_1(I_{n,k}), g_1(t))$$ and $$ \Delta_R(I_{n,k}, t) =
\Delta_R(g_1(I_{n,k}), g_1(t)).
$$

2. If $g_1$ decrease on $I_{n,k}$, then $$ \Delta_L(I_{n,k}, t) =
\Delta_R(g_1(I_{n,k}), g_1(t))$$ and $$ \Delta_R(I_{n,k}, t) =
\Delta_L(g_1(I_{n,k}), g_1(t)).$$
\end{lemma}

\begin{proof}
Lemma follows from the direct calculations by~\eqref{eq:1.5}.
\end{proof}

If $x\notin g_1^{-\infty}(0)$, then for any $n> n_0$ denote
$$\Delta_n(x) = \{ \Delta_L([\widehat{x}_n, \widehat{x}_n^{\,
+}], \Delta_R([\widehat{x}_n, \widehat{x}_n^{\, +}])\}.$$

If $g_1^{\widetilde{n}}(x)=0$, then for every $n>\max \{n_0,
\widetilde{n}\}$ denote $$\Delta_n(x) = \left\{
\Delta_R([\widehat{x}_n^{\, -}, \widehat{x}_n),
\Delta_L(\widehat{x}_n,\, \widehat{x}_n^{\, +}])\right\}.$$

The next lemma follows from Lemma~\ref{lema:5.16}.

\begin{lemma}\label{lema:5.17}
For any $n\geq n_0$, we have that $\Delta_n(x) =
\Delta_{n_0-1}(x)$.
\end{lemma}

We are now ready to prove Theorem~\ref{th:3}.

\begin{proof}[Proof of Theorem~\ref{th:3}]
Suppose that the derivative $h'(x)$ exists, is positive and
finite.

Since the derivative $h'(x)$ exists, is positive and finite, then
$\lim\limits_{n\rightarrow \infty}\Delta_n(x) = 0,$ whence it
follows from Lemma~\ref{lema:5.17} that $\Delta_{n_0-1}(x) =0$.
Thus, there exists $k,\, 0\leq k <2^{n_0-2}$ such that
$$\sup\limits_{t\in I_{n_0-1,k}}\Delta_R(I_{n_0-1,k},t)
=\sup\limits_{t\in I_{n_0-1,k}}\Delta_L(I_{n_0-1,k},t) =0,$$ which
means that $h$ is linear on $I_{n_0-1,k}$. Now piecewise linearity
of $h$ on the whole $[0,\, 1]$ follows from part 2. of
Lemma~\ref{lema:2.9}.

Part 2 of Theorem~\ref{th:3} follows from Lemmas~\ref{lema:5.14},
\ref{lema:5.15} and~\ref{lema:5.17}.
\end{proof}

\section{Length of the graph of the conjugacy}\label{sec:6}

We will prove Theorem~\ref{th:5} in this section. Let $h$ be the
conjugacy of firm carcass maps $g_1$ and $g_2$.

\begin{lemma}\label{lema:6.1}
Either $h$ is piecewise linear, or $h$ is differentiable at all
$g_1$-finite points.
\end{lemma}

\begin{proof}
Remind that, by the construction of $h_n$, we have that if $x\in
I_{n+1,k_n}$, then
$$ h_{n+1}'(x) = \frac{\# I_{n+1,k_n}(g_2)}{\# I_{n+1,k_n}(g_1)}
$$ for all $n\geq n_0$, whence, by Lemma~\ref{lema:4.22},
\begin{equation}\label{eq:6.1}
h_{n+1}'(x) = \frac{\# I_{n_0,k_{n_0-1}}(g_2)}{\#
I_{n_0,k_{n_0-1}}(g_1)}\cdot \prod\limits_{i=n_0+1}^{n}
\frac{\Rot^{x_i+x_{i-n_0+1}}(\delta_{n_0,p_{i-1}}(g_2))}
{\Rot^{x_i+x_{i-n_0+1}}(\delta_{n_0,p_{i-1}}(g_1))}
\end{equation}

Notice that if $x<1$ then~\eqref{eq:6.1} also determines
$h_{n+1}'(x+)$,

Suppose that $x$ is $g_1$-finite. Then there is $m>n_0$ such that
for all $i>m$ we have that
$$\Rot^{x_i+x_{i-n_0+1}}(\delta{n_0,p_{i-1}}) =
\delta_{n_0,0},$$ %
whence the multipliers in the infinite product
\begin{equation}\label{eq:6.2} R(x) = \frac{\# I_{n_0,k_{n_0-1}}(g_2)}{\#
I_{n_0,k_{n_0-1}}(g_1)}\cdot \prod\limits_{i=n_0+1}^{\infty}
\frac{\Rot^{x_i+x_{i-n_0+1}}(\delta_{n_0,p_{i-1}}(g_2))}
{\Rot^{x_i+x_{i-n_0+1}}(\delta_{n_0,p_{i-1}}(g_1))}
\end{equation} stabilize on
\begin{equation}\label{eq:6.3}\mathcal{P}
=\frac{\delta_{n_0,p_{i-1}}(g_2)}{\delta_{n_0,p_{i-1}}(g_1)}
\end{equation} and the limit~\eqref{eq:6.2} exists.

The left derivative $L(x)$ in the case $x$ is $g_1$-finite can be
found as \begin{equation}\label{eq:6.4} L(x) = \frac{\#
I_{n_0,k_{n_0-1}}(g_2)}{\# I_{n_0,k_{n_0-1}}(g_1)}\cdot
\prod\limits_{i=n_0+1}^{\infty}
\frac{\Rot^{y_i+y_{i-n_0+1}}(\delta_{n_0,q_{i-1}}(g_2))}
{\Rot^{y_i+y_{i-n_0+1}}(\delta_{n_0,q_{i-1}}(g_1))},
\end{equation} where \begin{equation}\label{eq:6.5}q_{i-1} =
\sum\limits_{j=i-n_0+1}^{i-1}\Rot^{y_{i-n_0}}(y_j)2^{i-1-j}
\end{equation} and the sequence $\{y_i, i\geq 1\}$ can be
constructed as follows:

1. $y_i = x_i$ for all $i$ before the beginning of the infinite
series of zeros in the $g_1$-expansion of~$x$.

2. For those $i$, where $x_i=0$ is an element of infinite series
of zeros take $y_i=1$.

\noindent Thus, for all $i>m$ we have that $q_{i-1} =
\sum\limits_{j=i-n_0+1}^{i-1}\Rot(1)2^{i-1-j} $ in~\eqref{eq:6.5},
i.e. $q_{i-1}=0$. Also the infinite product~\eqref{eq:6.4} for
$L(x)$ stabilizes on $\mathcal{P}$, which is defined
by~\eqref{eq:6.3}.

Now, by Theorem~\ref{th:3}, either $h$ is piecewise linear, or the
infinite products $L(x)$ and $R(x)$ exist.
\end{proof}

We are now ready to prove Theorem~\ref{th:5}.

\begin{proof}[Proof of Theorem~\ref{th:5}]
For any $g_1$-finite $x\in [0, 1]$ define $$ d_n =
\frac{h(\widehat{x}_n^{\, +}) -
h(\widehat{x}_n)}{\widehat{x}_n^{\, +} -\widehat{x}_n}.
$$ Since %
$\lim\limits_{n\rightarrow \infty} d_n = \lim\limits_{n\rightarrow
\infty}h_n'(x)$, then, by Theorem~\ref{th:3} and
Lemma~\ref{lema:6.1} either $\lim\limits_{n\rightarrow \infty} d_n
=0$, or $\lim\limits_{n\rightarrow \infty} d_n = +\infty$. In each
of these cases we have that $$ \lim\limits_{n\rightarrow \infty}
\frac{\sqrt{(h(\widehat{x}_n^{\, +}) - h(\widehat{x}_n))^2
+(\widehat{x}_n^{\, +} -\widehat{x}_n)^2}}{(\widehat{x}_n^{\, +}
-\widehat{x}_n) +(h(\widehat{x}_n^{\, +}) - h(\widehat{x}_n))} =
$$  $$= \lim\limits_{n\rightarrow \infty}\sqrt{\frac{1}{ 1
+\displaystyle{\frac{1}{ \displaystyle{\frac{h(\widehat{x}_n^{\,
+}) - h(\widehat{x}_n) }{ 2(\widehat{x}_n^{\, +} -\widehat{x}_n)}
+ \frac{ \widehat{x}_n^{\, +} -\widehat{x}_n}{
2(h(\widehat{x}_n^{\, +}) - h(\widehat{x}_n))}
 } }}}} =
$$$$ = \sqrt{\frac{1}{1+\frac{1}{0+\infty}}} =1.$$

Now, for every $\varepsilon>0$ and every $g_1$-finite $x\in (0,
1)$ there is an interval $I(x, \varepsilon) = (\widehat{x}_n,
\widehat{x}_n^{\, +})$ such that
\begin{equation}\label{eq:6.6}1 -\frac{\sqrt{(h(\widehat{x}_n^{\,
+}) - h(\widehat{x}_n))^2 +(\widehat{x}_n^{\, +}
-\widehat{x}_n)^2}}{(\widehat{x}_n^{\, +} -\widehat{x}_n)
+(h(\widehat{x}_n^{\, +}) - h(\widehat{x}_n))} < \varepsilon.
\end{equation}

Analogously for $x = 0$ and $n$ such that~\eqref{eq:6.6} holds
denote $I(0, \varepsilon) = [0, x_n^{\, +})$. For $x=1$ and $n$
such that~\eqref{eq:6.6} holds denote $I(1, \varepsilon) =
(\widehat{x}_n, 1]$. Notice that every $I(x, \varepsilon)$ is open
in $[0,\, 1]$ and $\bigcup\limits_xI(x, \varepsilon) = [0, 1], $
where the union is taken for all $g_1$-finite $x\in [0, 1]$.

Notice, that by construction for all $x_1,\, x_2$ the closed
intervals $\overline{I(x_1, \varepsilon)}$ and $\overline{I(x_2,
\varepsilon)}$ either do not intersect, or their intersection
consists of one point, or one of these sets is a subset of other.

It follows from the compactness of $[0, 1]$ that there exist
$x_1,\ldots, x_m$ such that
\begin{equation}\label{eq:6.7}\bigcup\limits_{i=1}^mI(x_i, \varepsilon) =
[0, 1],
\end{equation}

Now Theorem~\ref{th:5} follows from~\eqref{eq:6.6}
and~\eqref{eq:6.7}.
\end{proof}

\section{Skew tent maps}\label{sec:7}

The results about the conjugation of the carcase maps, which are
obtained in Sections~\ref{sec:5} and~\ref{sec:6}. Precisely, we
will reduce Theorem~\ref{th:4} from Theorem~\ref{th:3}.

\subsection{The derivative of the conjugation}

The following specification of Remark~\ref{rem:4.20} is clear for
the case of skew tent maps.

\begin{remark}\label{rem:8.1}
Let $g$ be the skew tent map with $g(v)=1$ and $x\in [0,\, 1]$ has
$g$-representation $\{(x_n,\, k_n),\ n\geq 1\}$. Then
$$ \# I_{n+1,k_n} = \# I_{n,k_{n-1}}\cdot \Rot^{x_n+x_{n+1}}(v).
$$ for any $n\geq 1$
\end{remark}
\begin{proof}
Since $I_{2,0} = [0, 1/2]$ and $I_{2,1} = [1/2, 1]$, denote $I_1 =
[0, 1]$.

By Remark~\ref{rem:4.13},
$$
\delta_{2,k} =\Rot^{k}(\delta_{1,0}) = \Rot^{k}(v)
$$ for any $k\in \{0; 1\}$

Thus, by Remark~\ref{rem:4.20}, for every $n\geq 2$ we have
$$ \delta_{n+1,k_n} =
\Rot^{x_{n-2}}(\delta_{2,\Rot^{x_{n-2}}(x_{n-1})}) =
\Rot^{x_{n-1}}(v).
$$

Now it follows from Remark~\ref{rem:4.10} that $$ \# I_{n+1,k_n} =
\# I_{n,k}\cdot \Rot^{x_n}(\delta_{n,k}) = \#
I_{n,k}\cdot\Rot^{x_n+x_{n+1}}(v).
$$
\end{proof}

The same result as in Remark~\ref{rem:8.1} can be obtained
directly from the basic notions (such as iteration of a function).
We show these reasonings in lemma below.

\begin{lemma}\label{lema:8.2}
For any $v\in (0,\, 1)$ and all $n\geq 1$ equalities $$
\mu_{n+1,4k+1}(f_v) = \mu_{n, 2k}(f_v) + v(\mu_{n,
2k+1}(f_v)-\mu_{n, 2k}(f_v)),
$$ and
$$
\mu_{n+1,4k+3}(f_v) = \mu_{n, 2k+1}(f_v) + (1-v)(\mu_{n,
2k+2}(f_v)-\mu_{n, 2k+1}(f_v))
$$
hold for all $k,\, 0\leq k<2^{n-2}$.
\end{lemma}

\begin{proof}
For every $m\geq 1$ and $t,\, 0\leq t\leq 2^{m-1}$ we will write
$\mu_{m,t}$ instead of $\mu_{m,t}(f_v)$.

Notice, that
$$\mu_{n+1,4k}<\mu_{n+1,4k+1}<\mu_{n+1,4k+2}<\mu_{n+1,4k+3}<\mu_{n+1,4k+4}$$
are consequent zeros of $f_v^{n+1}$,
$$\mu_{n+1,4k}<\mu_{n+1,4k+2}<\mu_{n+1,4k+4}
$$ are the consequent zeros of $f_v^n$, and $$
\mu_{n+1,4k}<\mu_{n+1,4k+4}
$$ are the consequent zeros of $f_v^{n-1}$.

Since $\mu_{n+1,4k+2}$ is zero of $f_v^n$, but is not zero of
$f_v^{n-1}$, then $f_v^{n-1}(\mu_{n+1,4k+2})=1$, whence
$M(\mu_{n+1,4k}, 0)$, $Q(\mu_{n+1,4k+2}, 1)$ and
$N(\mu_{n+1,4k+4}, 0)$ are the consequent kinks of the graph of
$f^{n-1}$ (see Fig.~\ref{fig-3}a.). Thus, the explicit formulas of
$f_v^{n-1}(x)$ for $x\in [\mu_{n+1,4k},\, \mu_{n+1,4k+4}]$ are
\begin{equation}\label{eq:8.1} f_v^{n-1}(x) = \left\{
\begin{array}{ll}
\displaystyle{\frac{x-\mu_{n+1,4k}}{\mu_{n+1,4k+2}-\mu_{n+1,4k}}}
& \text{if
}x\in [\mu_{n+1,4k},\, \mu_{n+1,4k+2}),\\
\\
\displaystyle{1-\frac{x-\mu_{n+1,4k+2}}{\mu_{n+1,4k+4}-\mu_{n+1,4k+2}}}
& \text{if }x\in (\mu_{n+1,4k+2}, \mu_{n+1,4k+4}].
\end{array}\right.\end{equation}

\begin{figure}[ht]
\begin{minipage}[h]{0.45\linewidth}
\begin{center}
\begin{picture}(140,120)
\put(0,100){\line(1,0){140}} \put(0,0){\line(1,0){140}}
\put(20,0){\line(1,2){50}} \put(70,100){\line(1,-2){50}}

\put(20,0){\line(0,1){100}} \put(120,0){\line(0,1){100}}

\put(20,0){\circle*{4}} \put(7,5){$M$} \put(120,0){\circle*{4}}
\put(123,5){$N$}

\put(70,100){\circle*{4}} \put(67,105){$Q$}
\put(20,100){\circle*{4}} \put(17,105){$P$}
\put(120,100){\circle*{4}} \put(117,105){$R$}
\end{picture}
\centerline{a. A part of the graph of $f_v^{n-1}$}\end{center}
\end{minipage}
\hfill
\begin{minipage}[h]{0.45\linewidth}
\begin{center}
\begin{picture}(140,120)
\put(0,100){\line(1,0){140}} \put(0,0){\line(1,0){140}}

\linethickness{0.3mm} \Vidr{20}{0}{70}{100} \VidrTo{120}{0}
\linethickness{0.1mm}

\Vidr{20}{0}{57.5}{100} \VidrTo{70}{0} \VidrTo{82.5}{100}
\VidrTo{120}{0}

\put(70,100){\circle*{4}} \put(67,105){$Q$}
\put(20,100){\circle*{4}} \put(17,105){$P$}
\put(120,100){\circle*{4}} \put(117,105){$R$}

\put(57.5,100){\circle*{4}} \put(54.5,105){$S$}
\put(82.5,100){\circle*{4}} \put(79.5,105){$T$}

\put(20,0){\circle*{4}} \put(7,5){$M$} \put(120,0){\circle*{4}}
\put(123,5){$N$}

\put(70,0){\circle*{4}} \put(75,5){$K$}
\end{picture}
\centerline{b. A part of the graph of $f_v^{n}$}\end{center}
\end{minipage}
\hfill \caption{Parts of the graph of iterations of  $f_v$}
\label{fig-3}
\end{figure}

Numbers $\mu_{n+1,4k+1}$ and $\mu_{n+1,4k+3}$ are solutions of
equations $f_v^{n+1}(\mu)=0$, $f_v^{n}(\mu)=1$ and
$f_v^{n-1}(\mu)=v$. The graph of $f_v^n$ is given at
Fig.~\ref{fig-3}b., where $S(\mu_{n+1,4k+1}, 1)$,
$K(\mu_{n+1,4k+2}, 0)$ and $T(\mu_{n+1,4k+3}, 1)$.

Thus, by~\eqref{eq:8.1}, $\mu_{n+1,4k+1}$ and $\mu_{n+1,4k+3}$ can
be found from equations $$
\frac{\mu_{n+1,4k+1}-\mu_{n+1,4k}}{\mu_{n+1,4k+2}-\mu_{n+1,4k}} =v
$$ and $$
1-\frac{\mu_{n+1,4k+3}-\mu_{n+1,4k+2}}{\mu_{n+1,4k+4}-\mu_{n+1,4k+2}}
= v,
$$
whence lemma follows from Remark~\ref{rem:3.3}.
\end{proof}

The next result follows from Remark~\ref{rem:8.1} (which is the
same as Lemma~\ref{lema:8.2}).

\begin{remark}\label{rem:8.3}
$$ \widehat{x}_{n+1}^{\, +} -\widehat{x}_{n+1} =
(\widehat{x}_n^{\, +} -\widehat{x}_n)\cdot \myae(x_n, x_{n+1}).
$$
\end{remark}

The next result follows from remarks~\ref{rem:5.8}
and~\ref{rem:8.3}.

\begin{remark}\label{rem:8.4}
For every $x\notin g_1^{-\infty}(0)$ we have that
\begin{equation}\label{eq:8.2} h_n'(x) = \prod\limits_{k=2}^n (2\,
\myae_v(x_{k}, x_{k-1})),
\end{equation} precisely
$$ L(x) = R(x) =
\prod\limits_{k=2}^\infty (2\, \myae_v(x_{k}, x_{k-1})),$$ Where
$L(x)$, $R(x)$ and $\myae_v$ are defined by~\eqref{eq:1.6},
\eqref{eq:1.7} and~\eqref{eq:1.8}.
\end{remark}

Now Theorem~\ref{th:4} follows from Theorem~\ref{th:3} and
Remark~\ref{rem:8.4}.

\subsection{Length of the graph of the conjugacy}

Now we will find the explicit formula for the length of the graph
of $h_{n+1}$, which approximates the conjugation $h$ of the tent
map $f$ and skew tent map $f_v$. The graph of $h_{n+1}$ divides
$[0,\, 1]$ into $2^n$ equal parts, where it is linear. At each of
these parts the derivative of $h_{n+1}$ equals
$\prod\limits_{i=2}^{n+1} \alpha_i$. Each of these $n$ multipliers
can be either $2v$, of $2(1-v)$.

Notice, that all the values $\alpha_i\in \{2v;\, 2-2v\}$ are
possible in the sequence $\alpha_2,\ldots, \alpha_{n+1}$
in~\eqref{eq:8.2} for $h_{n+1}'(x)$. Indeed, the map $h_{n+1}$ has
$2^n$ intervals of linearity and there are exactly $2^n$ choices
for the independent values of  $\alpha_2,\ldots \alpha_{n+1}$.
Thus, for every $k,\, 0\leq k\leq n$ there are $C_n^k$ intervals
of linearity of  $h_{n+1}$, where the derivative of $h_{n+1}$
equals $(2v)^k(2-2v)^{n-k}$.

Let $t$ be a tangent of the graph of $h_{n+1}$ on some of its
interval of linearity and $\alpha$ be the derivative of $h_{n+1}$
on its interval. Clearly, $\tan \alpha = t$. Then $\cos\alpha =
\sqrt{\frac{1}{1+t^2}}$ and the length of the graph on the
interval is $\frac{1}{2^n\cos\alpha} = \frac{1}{2^n}\sqrt{
1+t^2}$.

We can now express the length of the graph of $h_{n+1}$ on the
entire $[0, 1]$ as
\begin{equation}\label{eq:8.3}
l_{n+1}(v) = \frac{1}{2^{n}}\cdot\sum\limits_{k=0}^nC_n^k \cdot
\sqrt{1+2^{2n}v^{2k}(1-v)^{2(n-k)}}. \end{equation}

The following combinatorial fact follows from Theorem~\ref{th:5}.

\begin{lemma}\label{lema:8.5}
For every $v\in (0,\, 1)\backslash \{ 0.5\}$ the limit $
\lim\limits_{n\rightarrow \infty}l_n(v)=2$ holds, where $l_n(v)$
are defined by~(\ref{eq:8.3}).
\end{lemma}

Notice, that expression~(\ref{eq:8.3}) has cense also for $v\in \{
0;\, 0.5;\, 1\}$. Obviously, $l_n(0)=l_n(1)=1$ and
$l_n(0.5)=\sqrt{2}$. Moreover, the case $l_n(0.5)$ corresponds to
the trivial conjugation $y=x$ of the mapping $f$ with itself. We
have presented in~\cite[Ch.~13.1]{Plakh-Arx-Book} the numerical
calculations of $l_{n}(v)$, given by~\eqref{eq:8.3}, for different
$v\in (0,\, 1)$. We have calculated $l_{n}(v)$ up to so huge
values of $n$, that $l_{n}(v)> 1.97$.

Georgiy Shevchenko, professor of Taras Shevchenko National
University of Kyiv (Ukraine), noticed us that Lemma~\ref{lema:8.5}
can be simply proven with the use of reasonings, which are simple
for the specialists in probability theory.

\begin{proof}[Proof of Lemma~\ref{lema:8.5} by G. Shevchenko]
Clearly,
$$
l_n =  \int_0^1 \sqrt{1+ \big(h_n'(x)\big)^2} dx\leq \int_0^1
\big(1+ h_n'(x)\big) dx = 2,
$$
so it suffices to prove that $\liminf_{n\to\infty} l_n \geq 2$.

Let us consider the probability space $(\Omega,
\mathcal{F},\mathsf{P}) = ([0,1],\mathcal{B}([0,1]),\lambda)$,
where $\lambda$ is the Lebesgue measure. For $\omega\in[0,1]$,
define by $X_k(\omega)$, $k\geq 1$, the $k$th digit in the binary
representation of $\omega$, so that
$$
\omega = \sum_{k=1}^{\infty}2^{-k} X_k(\omega).
$$
Considered as random variables on $(\Omega, \mathcal{F},\mathsf{P})$, these digits are independent and have the  Bernoulli distribution $\mathrm{B}(1,\frac{1}{2})$, i.e. $\mathsf{P}(X_n = 0) = \mathsf{P}(X_n = 1) = \frac{1}{2}$. 

By Remark~\ref{rem:8.4}, $h_{n+1}'(\omega) =
\prod\limits_{k=2}^{n+1} \alpha_k(\omega)$ for almost all
$\omega\in \Omega$, where $\alpha_k(\omega) = 2v$ if
$X_{k}(\omega) = X_{k-1}(\omega)$ and $\alpha_k(\omega) = 2(1-v)$
otherwise; $X_0(\omega) = 0$. Then we can write
$$
h_n'(\omega) = 2^n v^{S_n(\omega)}(1-v)^{n-S_n(\omega)},
$$
where $S_n(\omega) = \sum\limits_{k=1}^n \ind{X_k(\omega) =
X_{k-1}(\omega)}$. Observe that the random variables $Y_k =
\ind{X_k(\omega) = X_{k-1}(\omega)}$ are independent and have the
Bernoulli distribution $\mathrm{B}(1,\frac{1}{2})$.

Then
\begin{gather*}
l_n = \ex{\sqrt{1+\big(h'_n(\omega)\big)^2}} = \ex{\sqrt{1+\prod_{k=2}^n \alpha_k(\omega)^2}} \\
= \ex{\sqrt{1+2^{2n} v^{2S_n(\omega)}(1-v)^{2n-2S_n(\omega)}}},
\end{gather*}
where %
$\ex{\cdot}$ denotes the expectation on $(\Omega,
\mathcal{F},\mathsf{P})$.

Now assume without loss of generality that $v>1/2$ and take any
number $b\in (1/2,v)$. Estimate
\begin{equation}\label{eq:8.4}
l_n \geq \mathsf{P}(S_n\leq bn) + 2^n\ex{
v^{S_n(\omega)}(1-v)^{n-S_n(\omega)}\ind{S_n>bn}}.
\end{equation}
Since the random variables $Y_k$, $k\geq 1$, are independent and
identically distributed, by the firm law of large numbers, $S_n/n
\to \ex{Y_1} = 1/2$, $n\to\infty$, almost surely. In particular,
$\mathsf{P}(S_n\leq bn)\to 1$, $n\to\infty$.

Further, define a measure $\mathsf{P} \ll \mathsf{P}$ on
$(\Omega,\mathcal{F})$ by
$$
\frac{d\mathsf{P}_n}{d\mathsf{P}}(\omega)  = h'_n(\omega) = 2^n
\prod_{k=1}^{n}v^{Y_k}(1-v)^{1-Y_k}.
$$
Since
$$
\ex{\frac{d\mathsf{P_n}}{d\mathsf{P}}(\omega)} = \ex{h'_n(\omega)}
= \int_0^1 h'_n(x) dx = 1,
$$
$\mathsf{P}_n$ is a probability measure.  Moreover, it is easy to
see that under $\mathsf{P}_n$ the random variables $Y_k$,
$k=1,\dots,n$ are independent and identically distributed  and
have the Bernoulli distribution $\mathrm{B}(1,v)$. Now
$$
2^n\ex{ v^{S_n(\omega)}(1-v)^{n-S_n(\omega)}\ind{S_n>bn}} = 2^n
\mathsf{E}_n \left[\ind{S_n>bn} \right] = \mathsf{P}_n (S_n>bn).
$$

Appealing to the firm law of large numbers once more,
$S_n/n\overset{\mathsf{P}_n}{\longrightarrow} v$, $n\to\infty$. In
particular, since $b<v$, $\mathsf{P}_n (S_n>bn)\to 1$,
$n\to\infty$. Consequently, in view of \eqref{eq:8.4}, we get
$\liminf\limits_{n\to\infty} l_n\geq 2$.
\end{proof}

\section{Hypothesis}\label{sec:8}

Theorems~\ref{th:4} and~\ref{th:5} can be also probed in the case,
when we change there a firm carcass map by a carcass map $g$, all
whose kinks are $g$-rational, i.e. the $g$-expansion of every kink
of $g$ is either $g$-finite, or is a periodical sequence of
numbers from the set $\{0;\, 1\}$.

We think that Theorems~\ref{th:4} and~\ref{th:5} are not true for
carcass maps in general, i.e. some additional assumptions about
the kinks of maps are necessary.

\setlength{\unitlength}{1pt}

\pagestyle{empty}
\bibliography{Ds-Bib}{}
\bibliographystyle{makar}

\newpage
\tableofcontents

\end{document}